\documentclass{article}

\usepackage[latin1]{inputenc}
\usepackage{amsthm,amssymb, amsmath}
\usepackage[all]{xy}
\usepackage{graphicx}
\usepackage[small]{caption}

\newcommand{\textaut}{\textsc}
\newcommand{\texttit}{\textit}
\newcommand{\textresto}{\textup}
\newcommand{\VSpc}[2]{\vrule height#1 width-1pt depth#2} 

\newtheorem{theorem}{Theorem}[section]
\newtheorem{definition}{Definition}[section]
\newtheorem{lemma}[theorem]{Lemma}
\newtheorem{proposition}[theorem]{Proposition}
\newtheorem{corollary}[theorem]{Corollary}
\newtheorem{conjecture}[theorem]{Conjecture}
\newtheorem*{example}{Example}

\newtheorem*{remark}{Remark}

\newcommand{\matr}[4]{
\left( \begin{array}{cc} #1 & #2 \\ #3 & #4 \end{array} \right)}

\begin{document}

\title{The entropy of $\alpha$-continued fractions:\\ numerical results}
\author{\begin{small}\textsc{Carlo Carminati, Stefano Marmi, Alessandro Profeti, Giulio Tiozzo}\end{small}}
\date{28 November, 2009}
\maketitle

\begin{abstract}
We consider the one-parameter family of interval maps arising from
generalized continued fraction expansions known as $\alpha$-continued
fractions. For such maps, we perform a numerical study of the
behaviour of metric entropy as a function of the parameter. 
The behaviour of entropy is known to be quite regular for parameters 
for which a \emph{matching condition} on the orbits of the endpoints holds. 
We give a detailed description of the set $\mathcal{M}$ where this condition is met: 
it consists of a countable union of open intervals, corresponding to different combinatorial 
data, which appear to be arranged in a hierarchical structure.
Our experimental data suggest that the complement of $\mathcal{M}$ is a proper subset of 
the set of bounded-type numbers, hence it has measure zero.
Furthermore, we give evidence that the entropy on matching intervals is smooth;
on the other hand, we can construct points outside of $\mathcal{M}$ on which it is not even locally monotone.
\end{abstract}

\section{Introduction}

Let $\alpha \in [0,1]$. We will study the one-parameter family of
one-dimensional maps of the interval
$$T_{\alpha} : [\alpha-1, \alpha] \rightarrow [\alpha-1, \alpha] $$
$$T_{\alpha}(x) = \left\{\begin{array}{ll} \frac{1}{|x|} - \left\lfloor \frac{1}{|x|} + 1 - \alpha \right\rfloor & \textup{if }x \neq 0 \\ 0 & \textup{if }x = 0 \end{array} \right.$$

If we let $x_{n, \alpha} = T_{\alpha}^n(x)$, $a_{n, \alpha} = \left\lfloor 
\frac{1}{|x_{n-1, \alpha}|} + 1 - \alpha \right\rfloor$, $\epsilon_{n, \alpha} = \textup{Sign}(x_{n-1, \alpha})$, 
then for every $x \in [\alpha-1, \alpha]$ we get the expansion
$$ x = \frac{\epsilon_{1, \alpha}}{a_{1, \alpha} + \frac{\epsilon_{2,
\alpha}}{a_{2, \alpha} +\phantom{x}\raisebox{-1.4ex}{\scalebox{0.8}{\hbox{$\ddots$}}}}}$$
with $a_{i, \alpha}\in \mathbb{N}, \epsilon_{i, \alpha} \in\{\pm 1\}$ which we call \emph{$\alpha$-continued fraction}. These systems were introduced by Nakada (\cite{Nakada81}) and are also known in the literature as $\emph{Japanese continued fractions}$.

The algorithm, analogously to Gauss' map in the classical case, provides rational approximations of real numbers. The convergents $\frac{p_{n, \alpha}}{q_{n, \alpha}}$ are given by

$$ \left\{ \begin{array}{ccc} p_{-1, \alpha} = 1 & p_{0, \alpha} = 0 & p_{n+1, \alpha} = \epsilon_{n+1, \alpha} p_{n-1, \alpha} + a_{n+1, \alpha}p_{n, \alpha} \\ q_{-1, \alpha} = 0 & q_{0, \alpha} = 1 & q_{n+1, \alpha} = \epsilon_{n+1, \alpha} q_{n-1, \alpha} + a_{n+1, \alpha}q_{n, \alpha} \end{array} \right.$$
It is known (see \cite{LuzziMarmi}) that for each $\alpha \in (0,1]$ there exists a unique invariant measure $\mu_{\alpha}(dx) = \rho_{\alpha}(x)dx$ absolutely continuous w.r.t. Lebesgue measure.

\begin{figure}[h]
\centering
\includegraphics[scale=0.4]{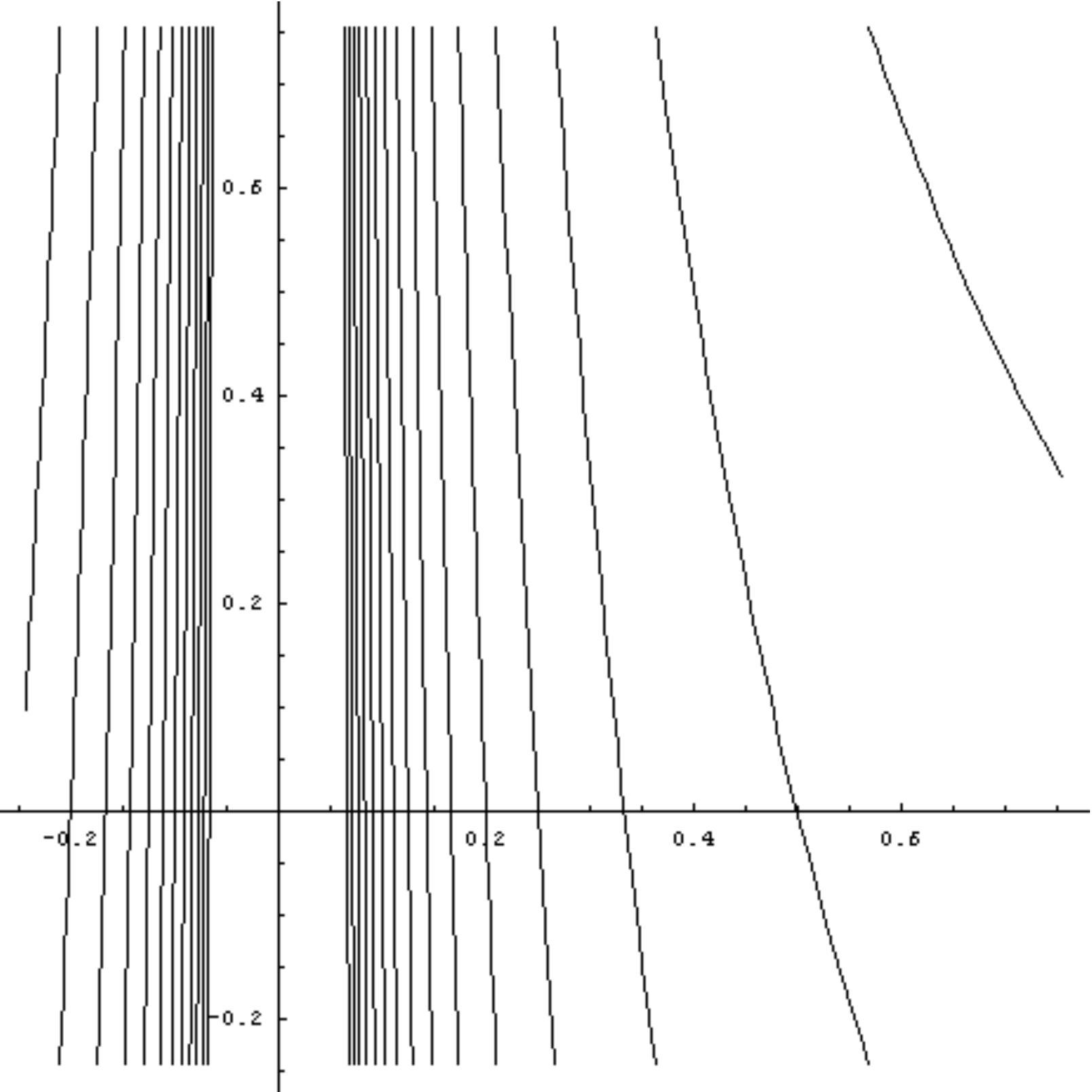}
\caption{Graph of $T_\alpha$} 
\end{figure}

In this paper we will focus on the metric entropy of the $T_{\alpha}$'s, which by Rohlin's formula (\cite{Rohlin}) is given by
$$ h(T_{\alpha}) = -2 \int_{\alpha-1}^{\alpha} \log |x| \rho_{\alpha}(x) dx $$
Equivalently, entropy can be thought of as the average exponential growth rate of the denominators of convergents:
for $\mu_{\alpha}$-a.e. $x \in [\alpha-1, \alpha]$,

$$h(T_{\alpha}) = 2 \lim_{n \rightarrow \infty} \frac{1}{n} \log q_{n, \alpha}(x)$$

The exact value of $h(T_{\alpha})$ has been computed for $\alpha \geq \frac{1}{2}$ by Nakada (\cite{Nakada81}) and for $\sqrt{2}-1 \leq \alpha \leq \frac{1}{2}$ by Cassa, Marmi and Moussa  (\cite{MarmiMoussaCassa}). 

In \cite{LuzziMarmi}, Luzzi and Marmi computed numerically the entropy for $\alpha \leq \sqrt{2}-1$
by approximating the integral in Rohlin's formula with Birkhoff averages
 
$$h(\alpha, N, x) = -\frac{2}{N} \sum_{j = 0}^{N-1} \log |T_{\alpha}^j(x)|$$
for a large number $M$ of starting points $x \in (\alpha-1, \alpha)$ and then averaging over the samples: 
$$h(\alpha, N ,M) = \frac{1}{M} \sum_{k = 1}^M h(\alpha, n , x_k)$$
Their computations show a rich structure for the behaviour of the
entropy as a function of $\alpha$; it seems that the function $\alpha
\mapsto h(T_\alpha)$ is piecewise regular and changes monotonicity on
different intervals of regularity.

These features have been confirmed by some results by Nakada and
Natsui (\cite{NakadaNatsui}, thm. 2) which give a \emph{matching condition}
on the orbits of $\alpha$ and $\alpha-1$ 
$$T^{k_1}_{\alpha}(\alpha) = T_{\alpha}^{k_2}(\alpha-1)\quad \textup{for some }k_1, k_2 \in \mathbb{N}$$
which allows to find countable families of intervals where the entropy is
increasing, decreasing or constant (see section \ref{matchingconditions}). It is not
difficult to check that the numerical data computed {\em via} Birkhoff
theorem fit extremely well with the matching intervals of \cite{NakadaNatsui}.
\begin{figure}[!h]
   \centering
   \includegraphics[scale=0.25]{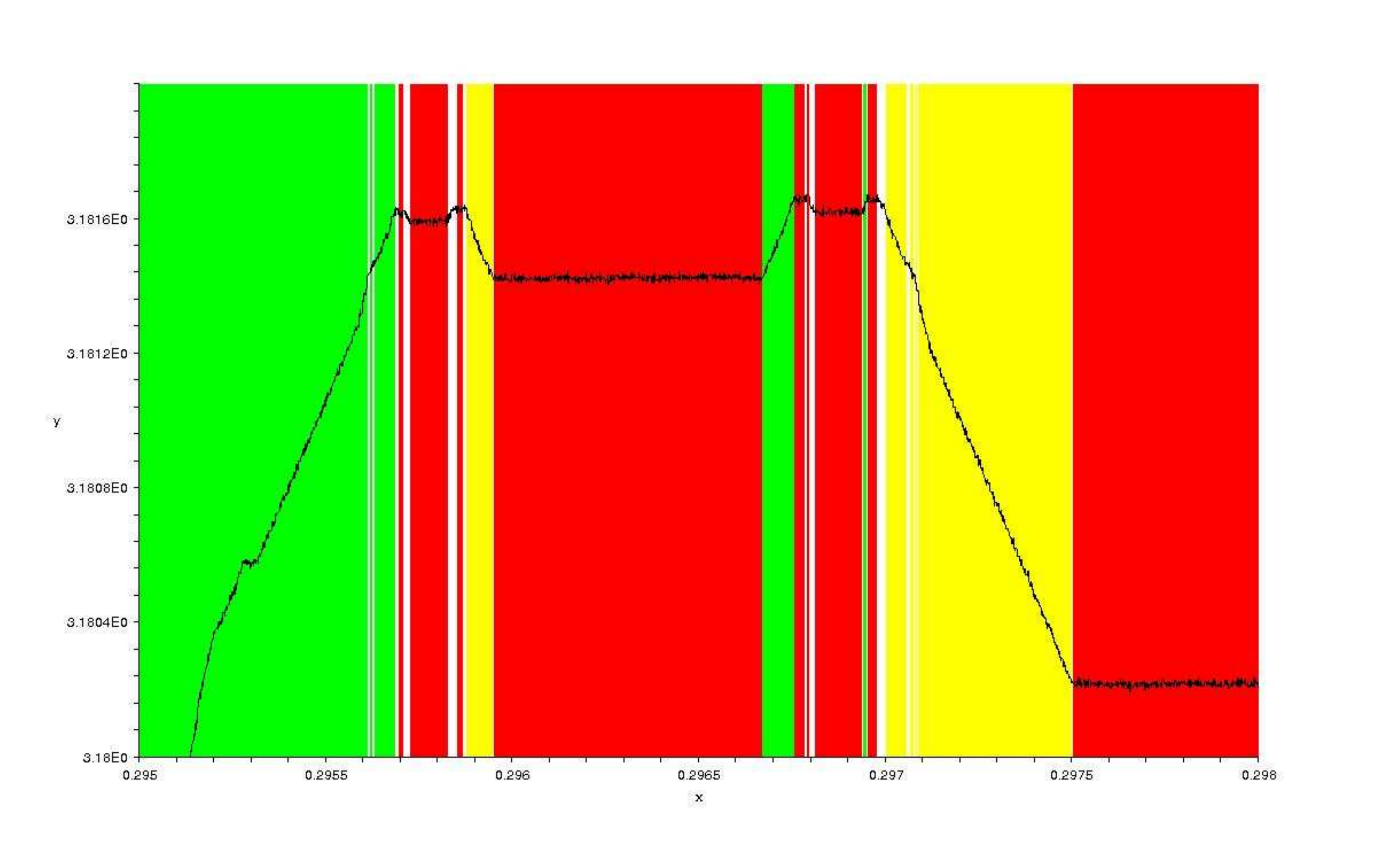}
   \caption{Numerical data vs. matching intervals}
   \label{MMvsNN}
  \end{figure}

In this paper we will study the matching condition in great detail. First of all, we analyze the
mechanism which produces it from a group-theoretical point
of view and find an algorithm to relate the $\alpha$-continued
fraction expansion of $\alpha$ and $\alpha-1$ when a matching
occurs. This allows us to understand the combinatorics behind the matchings
once and for all, without having to resort to specific matrix identities. As an example, 
we will explicitly construct a family of matching intervals which accumulate on a point different
from $0$. In fact we also have numerical evidence that there exist positive values, 
such as  $[0,3,\bar{1}]$, which are  cluster point for intervals of all the three matching types:
with $k_1<k_2$, $k_1=k_2$ and $k_1>k_2$.

We then describe an algorithm to produce a huge
quantity of matching intervals, whose exact endpoints can be found
computationally, and analyze the data thus obtained. These data show
matching intervals are organized in a hierarchical structure, and we
will describe a recursive procedure which should produce such structure. 

Let now $\mathcal{M}$ be the union of all matching intervals. 
It has been conjectured (\cite{NakadaNatsui}, sect. 4, pg. 1213) that $\mathcal{M}$ is an open, dense set of full Lebesgue measure. In fact, the correctness of our scheme would imply the following stronger 
\begin{conjecture} \label{nbounded}
For any $n$, all elements of $(\frac{1}{n+1}, \frac{1}{n}] \setminus \mathcal{M}$ have regular continued fraction expansion bounded by $n$.
\end{conjecture}
Since the set of numbers with bounded continued fraction expansion has Lebesgue measure zero, this clearly implies the previous conjecture.

We will then discuss some consequences of these matchings on the shape
of the entropy function, coming from a formula in
\cite{NakadaNatsui}. This formula allows us to recover the behaviour
of entropy in a neighbourhood of points where a matching condition is
present. First of all, we will use it to prove that entropy has
one-sided derivatives at every point belonging to some matching
interval, and also to recover the exact value of $h(T_\alpha)$ for
$\alpha \geq 2/5$. In general, though, to reconstruct the entropy one
also has to know the invariant density at one point.

As an example, we shall examine the entropy on an interval $J$ on which (by previous
experiments, see \cite{LuzziMarmi}, sect. 3) it was thought to be linearly
increasing: we numerically compute the invariant density for a single
value of $\alpha \in J$ and use it to predict the analytical form of the entropy on
$J$, which in fact happens to be not linear. The data produced with
this extrapolation method agree with high precision, and  much better than
any linear fit, with the values of $h(T_\alpha)$ computed via Birkhoff
averages.

The paper is structured as follows: in section 2 we will discuss
numerical simulations of the entropy and provide some theoretical
framework to justify the results; in section 3 we shall analyze the
mechanisms which produce the matching intervals and in section 4 we
will numerically produce them and study their hierarchical structure; in section 5 we will see how these
matching conditions affect the entropy function.

\section*{Acknowledgements}
This research was partially supported by the project ``Dynamical Systems and Applications'' of the Italian Ministry of University and Research\footnote{PRIN 2007B3RBEY.}, and the Centro di Ricerca Matematica  ``Ennio De Giorgi''.

\section{Numerical computation of the entropy}

Let us examine more closely the algorithm  used in $\cite{LuzziMarmi}$ to compute the entropy. A
numerical problem in evaluating Birkhoff averages  arises from the fact that 
 the orbit of a point
can fall very close to the origin: the computer will not distinguish a
very small value from zero. In this case we neglect this point, and
complete the (pseudo)orbit restarting from a  new random
seed\footnote{Another choice is to throw away the whole orbit and restart; it
  seems there is not much difference on the final result}. As a matter of fact this
algorithm produces an approximate value of

$$h_\epsilon(\alpha):=\int_{I_\alpha} f_\epsilon(x) d\mu_\alpha(x)
\ \ \ {\rm with} \ \ \ f_\epsilon(x):=\left\{
\begin{array}{ll}
0 & |x|\leq \epsilon\\
-2\log|x| & |x| > \epsilon
\end{array}
\right.
$$ where $\epsilon= 10^{-16}$; of course $h_\epsilon(\alpha)$ is an
excellent approximation of the entropy $h(\alpha)$, since the
difference is of order $\epsilon \log \epsilon^{-1}$. To calculate $h_{\epsilon}(\alpha)$ 
we use the Birkhoff sums
$$h_\epsilon(\alpha, N, x):=\frac{1}{N} \sum_{j=0}^{N-1}
f_\epsilon(T_{\alpha}^j(x))$$
and in $\cite{Tiozzo}$ the fourth author proves that for large $N$ the random
variable $h(\epsilon, N, \cdot)$ is distributed around its mean
$h_\epsilon(\alpha)$ approximately with  normal law and  standard deviation
$\sigma_\epsilon(\alpha)/\sqrt{N}$ where
$$ \sigma^2_\epsilon(\alpha):= \lim_{n\to +\infty}\int_{I_\alpha}
\left( \frac{S_n f_\epsilon- n \int f_\epsilon d \mu_{\alpha}}{\sqrt{n}}\right)^2 d\mu_\alpha$$
which explains the aforementioned result by Luzzi and Marmi $\cite{LuzziMarmi}$.

One of our goals is to study the function $\alpha \mapsto
\sigma^2_\epsilon(\alpha)$, in particular we ask whether it displays
some regularity like continuity or semicontinuity. To this aim we pushed the same scheme as in $\cite{LuzziMarmi}$ to get higher precision: 

\begin{enumerate}
\item
We take a sample of values $\alpha$ chosen in a particular subinterval  $J\subset [0,1]$;
\item For each value $\alpha$ we choose a random sample $\{x_1, ... , x_M\}$ in $I_\alpha$ (the cardinality M of this sample is usually $10^6$ or $10^7$);
\item For each $x_i\in I_\alpha$ ($i=1,...,M$) we evaluate
$h_\epsilon(\alpha, N, x_i)$ as described before
(the number of iterates N will be $10^4$);
\item
Finally, we evaluate the (approximate) entropy and take record of
standard deviation as well:
$$\hat h_\epsilon(\alpha,N,M):= \frac{1}{M}\sum_{i=1}^M
h_\epsilon(\alpha, N, x_i)$$
$$\hat \sigma_\epsilon(\alpha):=\sqrt{ \frac{1}{M}\sum_{i=1}^M
  [h_\epsilon(\alpha, N, x_i)- \hat h_\epsilon(\alpha,N,M)]^2}.$$
\end{enumerate}

\subsection{Central limit theorem}

Let us restate more precisely the convergence result for Birkhoff sums proved in \cite{Tiozzo}. Let
us denote by $BV(I_{\alpha})$ the space of real-valued,
$\mu_{\alpha}$-integrable, bounded variation functions of the interval
$I_{\alpha}$. We will denote by $S_nf$ the Birkhoff sum
$$S_n f = \sum_{j = 0}^{n-1} f \circ T_{\alpha}^j$$

\begin{lemma}
Let $\alpha \in (0,1]$ and $f$ be an element of $BV(I_{\alpha})$.
Then the sequence 
$$M_n = \int_{I_\alpha}\left(\frac{S_n f- n \int f d\mu_{\alpha}}{\sqrt{n}}\right)^2 d\mu_{\alpha}$$ converges to a real nonnegative value, which will be denoted by $\sigma^2$. Moreover, $\sigma^2 = 0$ if and only if there exists 
$u \in L^2(\mu_{\alpha})$ such that $u\rho_{\alpha} \in BV(I_{\alpha})$ and
\begin{equation} \label{cohom} 
f - \int_{I_\alpha} f d\mu_{\alpha} = u - u \circ T_{\alpha} 
\end{equation}
\end{lemma}

The condition given by $(\ref{cohom})$ is the same as in the proof of the central limit theorem for Gauss' map, and it's known as $\emph{cohomological equation}$. The main statement of the theorem is the following:

\begin{theorem}
Let $\alpha \in (0,1]$ and $f$ be an element of $BV(I_{\alpha})$ such that $(\ref{cohom})$ has no solutions. Then, for every $v \in \mathbb{R}$ we have 
$$ \lim_{n \rightarrow \infty} \mu_{\alpha}\left(\frac{S_n f - n \int_I f d\mu_{\alpha}}{\sigma \sqrt{n}} \leq v \right) = \frac{1}{\sqrt{2 \pi}}\int_{-\infty}^v e^{-x^2/2}dx$$
\end{theorem}

Since we know that the invariant density $\rho_{\alpha}$ is bounded from below by a nonzero constant, we can show that 

\begin{proposition}
For every real-valued nonconstant $f \in BV(I_{\alpha})$, the equation $(\ref{cohom})$ has no solutions. Hence, the central limit theorem holds.
\end{proposition}

Now, for every $\epsilon > 0$ the function $f_{\epsilon}$ define in the previous section is of bounded variation, hence the central limit theorem holds and the distribution of the approximate entropy $h_{\epsilon}(\alpha, N, \cdot)$ approaches a Gaussian when $N \rightarrow \infty$. 
As a corollary, for the standard deviation of Birkhoff averages 

$$\textup{Std}\left[ \frac{S_n f_{\epsilon}}{n} \right] = \mathbb{E}\left[ \left( \frac{S_n f_{\epsilon}}{n} - \int_{I_{\alpha}}f_{\epsilon}d\mu_{\alpha} \right)^2 \right]^{1/2} = \frac{\sigma}{\sqrt{n}} + o\left( \frac{1}{\sqrt{n}} \right)$$  

\subsection{Speed of convergence}

In terms of numerical simulations it is of primary importance to estimate the difference between the sum computed at the $n^{th}$ step and the asymptotic value: a semi-explicit bound is given by the following

\begin{theorem}
For every nonconstant real-valued $f \in BV(I_\alpha)$, there exists $C > 0$ such that
$$ \sup_{v \in \mathbb{R}}\left[ \mu_{\alpha} \left( \frac{S_nf - n \int_{I_\alpha} f d\mu_{\alpha}}{\sigma \sqrt{n}} \leq v \right) - \frac{1}{\sqrt{2 \pi}} \int_{-\infty}^v e^{-\frac{x^2}{2}} dx \right] \leq \frac{C}{\sqrt{n}} $$
\end{theorem}

\begin{proof}
It follows from a Berry-Ess\'een type of inequality. For details see (\cite{Broise}, th.8.1).                  
\end{proof}

\subsection{Dependence of standard deviation on $\alpha$}

Given these convergence results for the entropy, it is natural to ask how the standard deviation  varies with $\alpha$. In this case not a single exact value of $\sigma_{\epsilon}(\alpha)$ is known; by using the fact that natural extensions of $T_{\alpha}$ are conjugate (\cite{Kraaikamp}, \cite{NakadaNatsui}), it is straightforward to prove the 
\begin{lemma}
 The map $\alpha\mapsto \sigma(\alpha)$ is constant for $\alpha \in [\sqrt{2}-1, \frac{\sqrt{5}-1}{2}]$.
\end{lemma}

\begin{proof} See appendix.
\end{proof}
 
The numerical study of this quantity is pretty interesting. 
We first considered the window $J=[0.295, 0.304299]$, where the
entropy is non-monotone. On this interval the standard
deviation shows quite a strange behaviour: the values we have
recorded do not form a cloud clustering around a continuous line (like
for the entropy) but they cluster all above it.

\begin{figure}[h]
\centering
\includegraphics[scale=0.4]{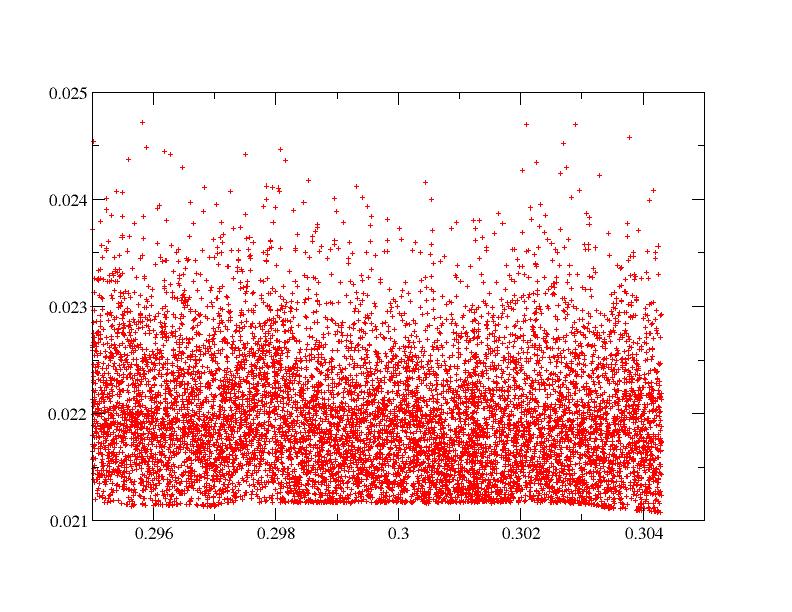}
\caption{Variance on the interval $J=[0.295, 0.304299]$. }
\end{figure}

One might guess that this is due to the fact that the map $\alpha\mapsto \sigma(\alpha)$ is only
semicontinuous, but the same kind of asymmetry takes place also on the
interval $J=[0.616, 0.618]$, where $\sigma^2$ is constant. Indeed, we can observe the same phenomenon also evaluating $\hat
\sigma_\epsilon(\alpha)$ for a fixed value $\alpha$ but taking several
different sample sets.
\begin{figure}[h]
\centering
\includegraphics[scale=0.4]{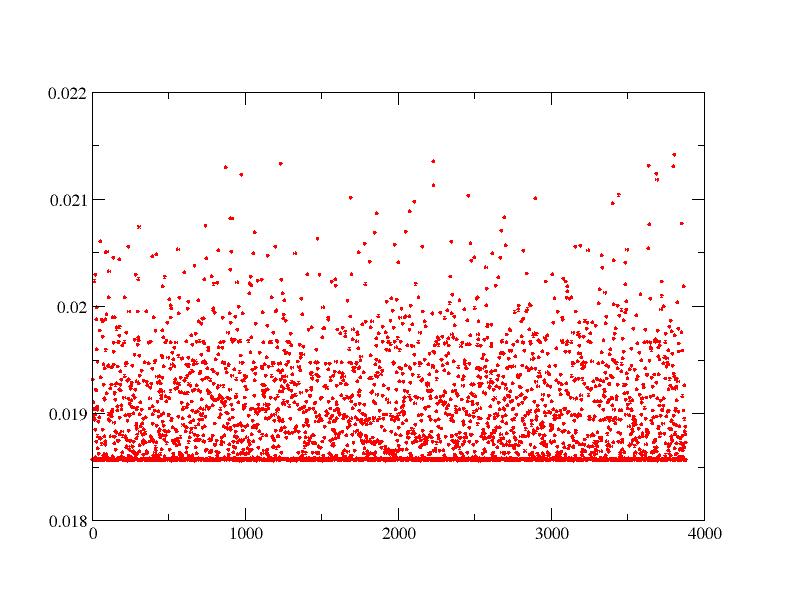}
\caption{Plot of the standard deviation of the different runs on the Gauss map}
\end{figure}

On the other hand this strange behaviour cannot be detected for other maps, 
like the logistic map, and could yet not be explained. Nevertheless, we point out that if you only consider $C^1$ observables, the standard deviation of Birkhoff sums can be proved continuous, at least for $\alpha \in (0.056, 2/3)$. See \cite{Tiozzo}.

\section{Matching conditions}\label{matchingconditions}

In \cite{NakadaNatsui}, Nakada and Natsui found a condition on the orbits of $\alpha$ and $\alpha-1$ which allows one to predict more precisely the behaviour of the entropy. Let us denote for any $\alpha \in [0,1]$, $x \in I_{\alpha}$, $n \geq 1$ by $M_{\alpha, x, n}$ the matrix such that $T_{\alpha}^n(x) = M^{-1}_{\alpha, x, n}(x)$, i.e.
$$M_{\alpha, x, n} = \matr{0}{\epsilon_{\alpha, 1}}{1}{c_{\alpha, 1}}\matr{0}{\epsilon_{\alpha, 2}}{1}{c_{\alpha, 2}}\dots \matr{0}{\epsilon_{\alpha, n}}{1}{c_{\alpha, n}}$$
They proved the following

\begin{theorem} \textup{(\cite{NakadaNatsui}, thm. 2)} \label{matchNakada}
Let us suppose that there exist positive integers $k_1$ and $k_2$ such that 
\begin{itemize}
 \item[\textup{(I)}] $\{T^n_{\alpha}(\alpha) : 0 \leq n < k_1 \} \cap \{T^m_{\alpha}(\alpha -1) : 0 \leq m < k_2 \} = \emptyset$
\item[\textup{(II)}] $M_{\alpha, \alpha, k_1} = \matr{1}{1}{0}{1} M_{\alpha, \alpha -1, k_2}$ \quad$[\, \Longrightarrow \;T^{k_1}_{\alpha}(\alpha) = T^{k_2}_{\alpha}(\alpha-1)\,]$
\item[\textup{(III)}] $T^{k_1}_{\alpha}(\alpha)\;\,\left[\,=\; T^{k_2}_{\alpha}(\alpha -1)\,\right]\; \notin \;\{\alpha, \alpha -1\}$
\end{itemize}
Then there exists $\eta > 0$ such that, on $(\alpha - \eta, \alpha + \eta)$, $h(T_{\alpha})$ is :

\begin{itemize}
\item[\textup{(i)}] strictly increasing if $k_1 < k_2$
\item[\textup{(ii)}] constant if $k_1 = k_2$
\item[\textup{(iii)}] strictly decreasing if $k_1 > k_2$
\end{itemize}
\end{theorem}

It turns out that conditions (I)-(II)-(III) define a collection of
open intervals (called \emph{matching intervals}); they also proved
that each of the cases (i), (ii) and (iii) takes place at least on one
infinite family of disjoint matching intervals clustering at the
origin, thus proving the non-monotonicity of the entropy
function. Moreover, they conjectured that the union of all matching
intervals is a dense, open subset of $[0,1]$ with full Lebesgue
measure.

In the following we will analyze more closely the mechanism which
leads to the existence of such matchings. As a consequence, we shall
see that it looks more natural to drop condition (III) from the
previous definition and replace (II) with

\begin{small}
$$\textup{(II$'$)} \qquad M_{\alpha, \alpha, k_1-1} = \pm \matr{1}{1}{0}{1}
  M_{\alpha, \alpha-1, k_2-1} \matr{1}{0}{-1}{-1}$$
\end{small}
(which implies $\frac{1}{T^{k_1-1}(\alpha)} + \frac{1}{T^{k_2-1}(\alpha-1)} = -1$).

\vskip 10 pt
\noindent 
We can now define the \emph{matching set} as
$$\mathcal{M} = \{ \alpha \in (0,1] \textup{ s. t. (I) and (II$'$) hold } \}$$ 
Note $\mathcal{M}$ is open, since the symbolic codings of $\alpha$ up to step $k_1-1$ and of $\alpha-1$ up
to step $k_2-1$ are locally constant.

Moreover, we will see that under this condition it is possible to
predict the symbolic orbit of $\alpha-1$ given the symbolic orbit of
$\alpha$, and viceversa. As an application, we will construct a
countable family of intervals which accumulates in a point different
from $0$.

Let us point out that our definition of matching produces a set slightly
bigger than the union of all matching intervals satisfying condition
(I,II,III): in fact the difference is just a countable set of points.

\subsection{Structure of PGL(2, $\mathbb{Z}$)}

Let us define $PGL(2, \mathbb{Z}) := GL(2, \mathbb{Z})/\{\pm I\}$, $PSL(2, \mathbb{Z}) := SL(2, \mathbb{Z})/\{ \pm I \}$. We have an exact sequence
$$1 \rightarrow PSL(2, \mathbb{Z}) \rightarrow PGL(2, \mathbb{Z}) \rightarrow \{ \pm 1\} \rightarrow 1$$
where the first arrow is the inclusion and the second the determinant; moreover,
if we consider the group 
$$C = \left\{ \matr{\pm 1}{0}{0}{\pm 1} \right\} \cong \frac{\mathbb{Z}}{2\mathbb{Z}} \times \frac{\mathbb{Z}}{2\mathbb{Z}}$$
and let $\overline{C} = C /\{ \pm I \}$, then
$$\overline{C} \cap PSL(2, \mathbb{Z}) = \{e\}$$
therefore we have the semidirect product decomposition
$$PGL(2, \mathbb{Z}) = PSL(2, \mathbb{Z}) \rtimes \overline{C}$$
Now, it is well known that $PSL(2, \mathbb{Z})$ is the free product
$$PSL(2, \mathbb{Z}) = <S> \star <U>$$
where 
$$S = \matr{0}{-1}{1}{0} \qquad U = \matr{0}{-1}{1}{1}$$
are such that $S^2 = I$, $U^3 = I$. 
Geometrically, $S$ represents the function $\{z \rightarrow -\frac{1}{z}\}$,
and if we denote by $T$ the element corresponding to the translation $\{ z \rightarrow z + 1\}$, we have $U = ST$.

The matrix $V = \matr{-1}{0}{0}{1}$ projects to a generator of $\overline{C}$ and it satisfies $V^2 = I$, $VSV^{-1} = VSV = S$ and $VTV^{-1} = T^{-1}$ in $PGL(2, \mathbb{Z})$ 
so we get the presentation
\begin{small}
$$PGL(2, \mathbb{Z}) = \{S, T, V \mid S^2 = I, (ST)^3 = I, V^2 = I, VSV^{-1} = S, VTV^{-1} = T^{-1} \}$$
\end{small}

\subsection{Encoding of matchings}

Every step of the algorithm generating $\alpha$-continued fractions consists of an operation of the type: 
$$ z\mapsto \frac{\epsilon}{z} - c \qquad \epsilon \in \{ \pm 1\}, c \in \mathbb{N}$$ which corresponds to the matrix $T^{-c}SV^{e(\epsilon)}$
with 
$$ e(\epsilon) = \left\{ \begin{array}{ll} 0 &\textup{\  if \ } \epsilon = -1 \\
                                           1 &\textup{\  if \ } \epsilon = 1 \end{array} \right.$$
so if $x$ belongs to the cylinder $((c_1, \epsilon_1), \dots, (c_k, \epsilon_k))$
we can express
$$T_{\alpha}^k(x) = T^{-c_k}SV^{e(\epsilon_k)}\dots T^{-c_1}SV^{e(\epsilon_1)}(x)$$

Now, suppose we have a matching $T^{k_1}_{\alpha}(\alpha) = T^{k_2}_{\alpha}(\alpha-1)$ and let $\alpha$ belong to the cylinder $((a_1, \epsilon_1), \dots, (a_{k_1}, \epsilon_{k_1}))$ and $\alpha -1$ belong to the cylinder 
$((b_1, \eta_1), \dots, (b_{k_2}, \eta_{k_2}))$. One can rewrite the matching condition as

$$T^{-a_{k_1}}SV^{e(\epsilon_{k_1})}\dots T^{-a_1}SV^{e(\epsilon_1)}(\alpha) = T^{-b_{k_2}}SV^{e(\eta_{k_2})}\dots T^{-b_1}SV^{e(\eta_1)}T^{-1}(\alpha)$$
hence it is sufficient to have an equality of the two M\"obius transformations
$$T^{-a_{k_1}}SV^{e(\epsilon_{k_1})}\dots T^{-a_1}SV^{e(\epsilon_1)} = T^{-b_{k_2}}SV^{e(\eta_{k_2})}\dots T^{-b_1}SV^{e(\eta_1)}T^{-1}$$
We call such a matching an \emph{algebraic matching}.
\noindent Now, numerical evidence shows that, if a matching occurs, then 
$$ 
\begin{array}{cll}
\epsilon_1 = +1    && \\[2pt]
\epsilon_{i} = -1  &\;& \textup{ for }2 \leq i \leq k_1-1 \\[2pt]
\eta_i = -1        &\;&\textup{ for }1 \leq i \leq k_2-1 
\end{array}
$$
If we make this assumption we can rewrite the matching condition as
\begin{small}
$$V^{e(\epsilon_{k_1})+1} T^{a_{k_1}(-1)^{e(\epsilon_{k_1})}}ST^{a_{k_1-1}}S \cdots T^{a_1}S = $$
$$= V^{e(\eta_{k_2})} T^{b_{k_2}(-1)^{[e(\eta_{k_2})+1]}}ST^{-b_{k_2-1}}S \cdots T^{-b_1} S T^{-1}$$
\end{small}
which implies $e(\epsilon_{k_1}) = e(\eta_{k_2}) +1$, i.e. $\epsilon_{k_1} \eta_{k_2} = -1$. 
\noindent If for instance $e(\epsilon_{k_1})= 1$ and $e(\eta_{k_2}) = 0$, by substituting $T = SU$ one has

$$(U^2S)^{a_{k_1}} U (SU)^{a_{k_1-1}-2} SU^2 \dots SU^2 (SU)^{a_1-2} SUS = $$
$$ = (U^2S)^{b_{k_2}-1} US (U^2S)^{b_{k_2-1}-2} US \dots US (U^2S)^{b_1-2} US$$

Since every element of $PSL(2, \mathbb{Z})$ can be written as a product of $S$ and $U$ in a unique way, one can get a relation between the $a_r$ and $b_r$. Notice that, since we are interested in $\alpha \leq \sqrt{2}-1$, $a_i \geq 2$ and $b_i \geq 2$ for every $i$, hence there is no cancellation in the equation above. By counting the number of $(U^2S)$ blocks at the beginning of the word, one has $a_{k_1} = b_{k_2}-1$, and by semplifying,

\begin{equation}  \label{matchalg}
\begin{array}{c} (SU)^{a_{k_1-1}-2} SU^2 \dots SU^2 (SU)^{a_1-2} SUS = \\
= S (U^2S)^{b_{k_2-1}-2} US \dots US (U^2S)^{b_1-2} US \end{array} 
\end{equation}
If one has $e(\epsilon_{k_1}) = 0$ and $e(\eta_{k_2}) = 1$ instead, the matching condition is 
$$(SU)^{a_{k_1}-1}SU^2 (SU)^{a_{k_1-1}-2} SU^2 \dots SU^2 (SU)^{a_1-2} SUS = $$
$$ (SU)^{b_{k_2}} SU^2 S (U^2S)^{b_{k_2-1}-2} US \dots US (U^2S)^{b_1-2} US $$
which implies $b_{k_2} = a_{k_1}-1$, and simplifying still yields equation $(\ref{matchalg})$.

Let us remark that $(\ref{matchalg})$ is equivalent to
$$ T^{-1} S T^{-a_{k_1-1}} S \dots T^{-a_1}SV = VS T^{-b_{k_2-1}} S \dots T^{-b_1} S T^{-1} $$
which is precisely condition $(\mathrm{II}')$: by evaluating both sides on $\alpha$

$$\frac{1}{T^{k_1-1}(\alpha)} + \frac{1}{T^{k_2-1}(\alpha-1)} = -1 $$

 Moreover, from $(\ref{matchalg})$ one has that to every $a_r$ bigger than $2$ it corresponds exactly a sequence of $b_i = 2$ of length precisely $a_r-2$, and viceversa. More formally, one can give the following algorithm to produce the coding of the orbit of $\alpha-1$ up to step $k_2-1$ given the coding of the orbit of $\alpha$ up to step $k_1-1$ (under the hypothesis that an algebraic matching  occurs, and at least $k_1$ is known). 

\begin{enumerate}
\item Write down the coding of $\alpha$ from step $1$ to $k_1-1$, separated by a symbol~$\star$ 
$$a_1 \star a_2 \star \dots \star a_{k_1-1}$$
\item Subtract $2$ from every $a_r$; if $a_r = 2$, then leave the space empty instead of writing $0$.
$$a_1-2 \star a_2-2 \star \dots \star a_{k_1-1}-2$$
\item Replace stars with numbers and viceversa (replace the number $n$ with $n$ consecutive stars, and write the number $n$ in place of $n$ stars in a row)
\item Add $2$ to every number you find and remove the stars: you'll get the sequence $(b_1, \dots , b_{{k_2}-1})$. 
\end{enumerate}

\begin{example}
Let us suppose there is a matching with $k_1 = n+3$ and $\alpha$ has initial coding $((3, +), (4, -)^n, (2, -))$. The steps of the algorithm are: 
\begin{itemize}
\item[Step 1] $$3\star\underbrace{4\star4\star\dots\star4\star}_{n\ times}2$$
\item[Step 2] $$1 \star\underbrace{2\star2\star\dots\star2\star}_{n\ times}$$
\item[Step 3] $$\star1\underbrace{\star\star1\star\star 1\dots 1 \star \star 1}_{n \ times}$$
\item[Step 4] $$2\ 3\ \underbrace{2 \ 3 \ \dots 2 \ 3}_{n \ times}$$
\end{itemize}
so the coding of $\alpha-1$ is $((2, -)(3, -))^{n+1}$, and $k_2 = 2n+3$.

\end{example}

\subsection{Construction of matchings}

Let us now use this knowledge to construct explicitly an infinite family of matching intervals which accumulates on a non-zero value of $\alpha$. For every $n$, let us consider the values of $\alpha$ such that $\alpha$ belongs to the cylinder $((3, +),(4,-)^n, (2,-))$ with the respect to $T_{\alpha}$. Let us compute the endpoints of such a cylinder.

\begin{itemize}
 \item The right endpoint is defined by 
$$\matr{-4}{-1}{1}{0}^n\matr{-3}{1}{1}{0}(\alpha) = \alpha-1$$ i.e.
$$\matr{1}{1}{0}{1}\matr{-4}{-1}{1}{0}^n\matr{-3}{1}{1}{0}(\alpha) = \alpha$$
\item The left endpoint is defined by
$$\matr{-4}{-1}{1}{0}^n\matr{-3}{1}{1}{0}(\alpha) = -\frac{1}{\alpha+2}$$ i.e.
$$\matr{-2}{-1}{1}{0}\matr{-4}{-1}{1}{0}^n\matr{-3}{1}{1}{0}(\alpha) = \alpha$$
\end{itemize}

By diagonalizing the matrices and computing the powers one can compute these value explicitly. In particular, 

$$\alpha^1_{min} = \frac{\sqrt{3}-1}{2} + \frac{40\sqrt{3}-69}{13}(2+\sqrt{3})^{-2n} + O((2+\sqrt{3})^{-4n})$$
$$\alpha^1_{max} = \frac{\sqrt{3}-1}{2} + \frac{10\sqrt{3}-12}{13}(2+\sqrt{3})^{-2n} + O((2+\sqrt{3})^{-4n})$$

The $\alpha$s such that $\alpha-1$ belongs to the cylinder $((2,-),(3,-))^{n+1}$ are defined by the equations

$$\left[ \matr{-3}{-1}{1}{0} \matr{-2}{-1}{1}{0} \right]^{n+1}(\alpha-1) = \alpha -1 $$
for the left endpoint and
$$\left[ \matr{-3}{-1}{1}{0} \matr{-2}{-1}{1}{0} \right]^{n+1}(\alpha-1) = \alpha $$
for the right endpoint, so the left endpoint corresponds to the periodic point such that $$\left[ \matr{-3}{-1}{1}{0} \matr{-2}{-1}{1}{0} \right](\alpha-1) = \alpha-1$$ i.e. 
$$\alpha^2_{min} = \frac{\sqrt{3}-1}{2}$$
and

$$\alpha^2_{max} = \frac{\sqrt{3}-1}{2} + \frac{33-19\sqrt{3}}{2} (2+\sqrt{3})^{-2n} + O((2+\sqrt{3})^{-4n})$$
By comparing the first order terms one gets asymptotically

$$\alpha^2_{min} < \alpha^1_{min} < \alpha^2_{max} < \alpha^1_{max}$$
hence the two intervals intersect for infinitely many $n$, producing infinitely many matching intervals which accumulate at the point $\alpha_0 = \frac{\sqrt{3}-1}{2}$. The length of such intervals is 
$$\alpha^2_{max}-\alpha^1_{min} = \frac{567-327\sqrt{3}}{26}(2+\sqrt{3})^{-2n} + O((2+\sqrt{3})^{-4n})$$

\section{Numerical production of matchings} \label{CAmatching}

In this section we will describe an algorithm to produce a lot of
matching intervals (i.e. find out their endpoints exactly), as well as
the results we obtained through its implementation. Our first attempt to find
matching intervals used the following scheme:

\begin{enumerate}
\item We generate a random seed of values $\alpha_i$ belonging to $ [0,1]$ (or some other interval of interest). When a high precision is needed (we manage to detect intervals of size $10^{-60}$) the random seed is composed by algebraic numbers, in order to allow symbolic (i.e. non floating-point) computation. 

 \item We find numerically candidates for the values of $k_1$ and
   $k_2$ (if any) simply by computing the orbits of $\alpha$ and of
   $\alpha-1$ up to some finite number of steps, and numerically
   checking if $T_\alpha^{k_1}(\alpha) = T_\alpha^{k_2}(\alpha-1)$ holds approximately for some $k_1$
   and $k_2$ smaller than some bound.

 \item Given any triplet $(\bar\alpha, k_1, k_2)$ determined as above,  we compute  the
   symbolic orbit of $\bar\alpha$ up to step $k_1-1$ and the orbit of
   $\bar\alpha-1$ up to step $k_2-1$. 

\item We check that the two M\"obius transformations associated to these symbolic orbits satisfy condition $(\mathrm{II}')$:

$$M_{\alpha, \alpha, k_1-1} = \pm \matr{1}{1}{0}{1} M_{\alpha, \alpha-1, k_2-1} \matr{1}{0}{-1}{-1} $$ 

\item We solve the system of quadratic equations which correspond to
  imposing that $\alpha$ and $\alpha-1$ have the same symbolic orbit
  as $\bar\alpha$ and $\bar\alpha-1$, respectively.

Let us remark that this is the heaviest step of the whole procedure since we must
  solve $k_1+k_2-2$ quadratic inequalities; for this reason the value
  $k=k_1+k_2$ may be thought of as a measure of the computational cost
  of the matching interval and will be referred to as {\it order of
    matching}.
\end{enumerate}

Following this scheme, we detected more than $10^7$ matching intervals, whose endpoints are quadratic surds; their union still leaves many gaps, each of which smaller than $6.6\cdot 10^{-6}$. A table with a sample of such data is contained in the appendix. \footnote{A more efficient algorithm, which avoids random sampling, will be discussed in subsection~\ref{bintree}.}

In order to detect some patterns in the data, let us plot the size of these intervals (figure \ref{ms}). For each matching interval $]\alpha_-, \alpha_+[$, we drew the point of coordinates $(\alpha_-, \alpha_+-\alpha_-)$.

 \begin{figure}[!h]
   \centering
\includegraphics[scale=0.3]{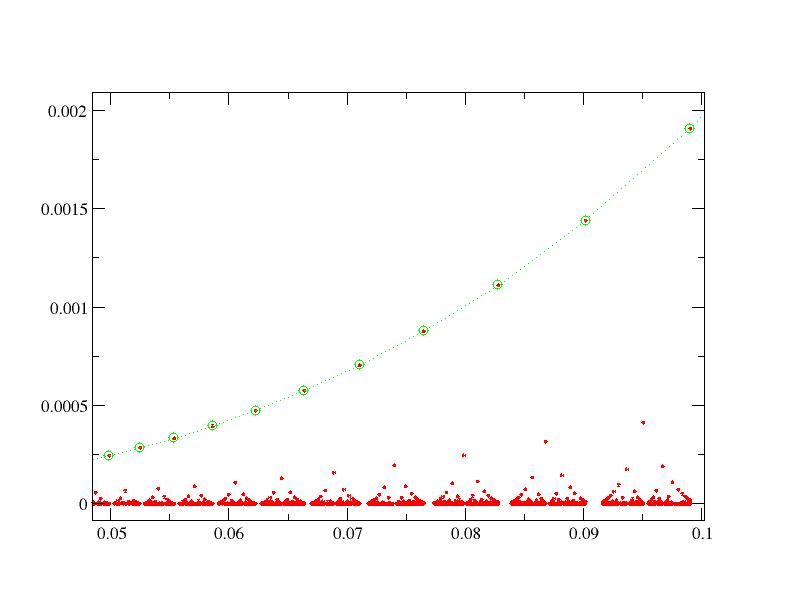}
   \caption{Size of matchings}
\label{ms}
  \end{figure}

It seems there is some self-similar pattern: in order to understand
better its structure it is useful to identify some ``borderline''
families of points. The most evident family is the one that appears as the
higher line of points in the above figure (which we have highlighted in green): these points correspond
to matching intervals which contain the values $1/n$, and their endpoints are $
\alpha_-(n)=\frac{1}{2}[\sqrt{n^2+4}-n]$, $\alpha_+(n)=
\frac{1}{2n-2}[\sqrt{n^2+2n-3}-n+1]$; this is the family $I_n$ already
exhibited in \cite{NakadaNatsui}. Since $\alpha_-(n)=
1/n-1/n^3+o(1/n^3)$ and $\alpha_+(n)= 1/n+1/n^3+o(1/n^3)$, for $n\gg 1$
the points $(\alpha_-(n), \alpha_+(n)-\alpha_-(n))$ are very close to
$(\frac{1}{n},\frac{1}{n^3})$. This suggests that this family will
``straighten'' if we replot our data in log-log scale. This is indeed
the case, and in fact it seems that there are also other families
which get perfectly aligned along parallel lines of slope 3 (see figure \ref{loglogscale}). 

\begin{figure}[!h] 
   \centering
\includegraphics[scale=0.3]{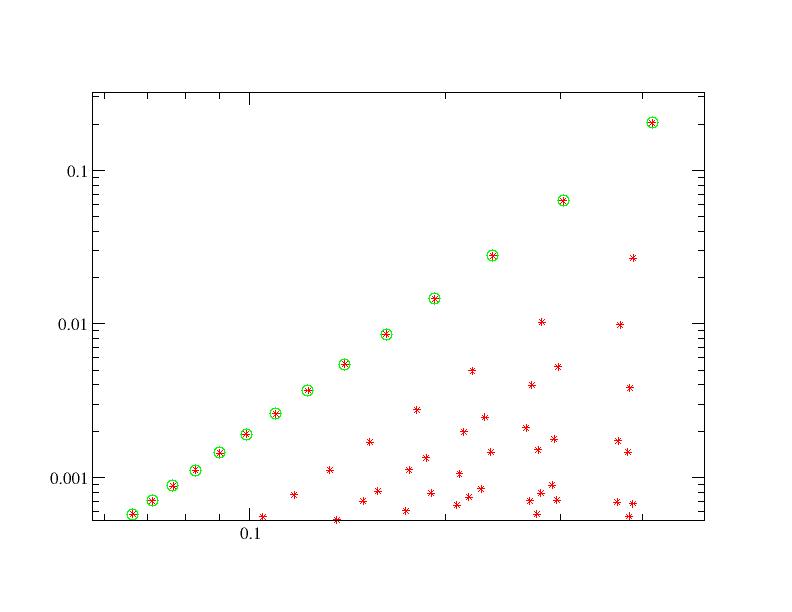}
   \caption{Same picture, in log-log  scale.}
\label{loglogscale}  
\end{figure}

If we consider the ordinary continued fraction expansion of the elements of these families we realize
that they obey to some very simple empirical\footnote{Unfortunately we are still not able to prove all these rules.} rules: 
\begin{enumerate}
\item[(i)] the endpoints of any matching interval have a purely
  periodic continued fraction expansion of the type $[0,\overline{a_1,
      a_2, ..., a_m,1}]$ and $[0,\overline{a_1, a_2, ..., a_m+1}]$;
this implies that the rational number corresponding to $[0,a_1, a_2,
    ..., a_m+1]$ is a common convergent of both endpoints and is the rational with 
smallest denominator which falls inside the matching interval;
\item[(ii)]
any endpoint $[0,\overline{a_1, a_2, ..., a_m}]$  of
a matching interval belongs to a  family $\{ [0,\overline{a, a_2, ..., a_m}] \ : \ a\geq \max_{2\leq i\leq m}a_i\}$; in particular this family has a member in each cylinder $B_n:= \{\alpha \ : 1/(n+1)<\alpha<1/n\}$ for $n \geq a$, so that each family will cluster at the origin.
\item[(ii')] other families can be detected in terms of the continued
  fraction expansion: for instance on each cylinder $B_n$ ($n\geq
  3)$ the largest matching interval on which $h$ is decreasing has
  endpoints with expansion $[0,\overline{n,2,1,n-1,1}]$ and $[0,\overline{n,2,1,n}]$
\item[(iii)]
matching intervals seem to be organized in a binary tree structure, which is related to the Stern-Brocot tree\footnote{Sometimes also known as Farey tree. See \cite{Knuth}.}: 
one can thus design a bisection algorithm to fill in the gaps between intervals, and what it's left over is a closed, nowhere dense set. This and the following points will be analyzed extensively in subsection~\ref{bintree};
\item[(iv)] if $\alpha\in B_n$ is the endpoint of some matching
  interval then $\alpha=[0;\overline{a_1, a_2, ..., a_m}]$ with
  $a_i\leq n \ \forall i\in \{1,...,m\}$; this would imply that the values $\alpha \in B_n$ which do not belong to any  matching interval must be bounded-type numbers with
  partial quotients bounded above by $n$;
\item[(v)] it is possible to compute the exponent $(k_1,k_2)$ of a matching from the continued fraction expansion of any one of its endpoints.
\end{enumerate}

\begin{figure}[!h]   \centering
\includegraphics[scale=0.25]{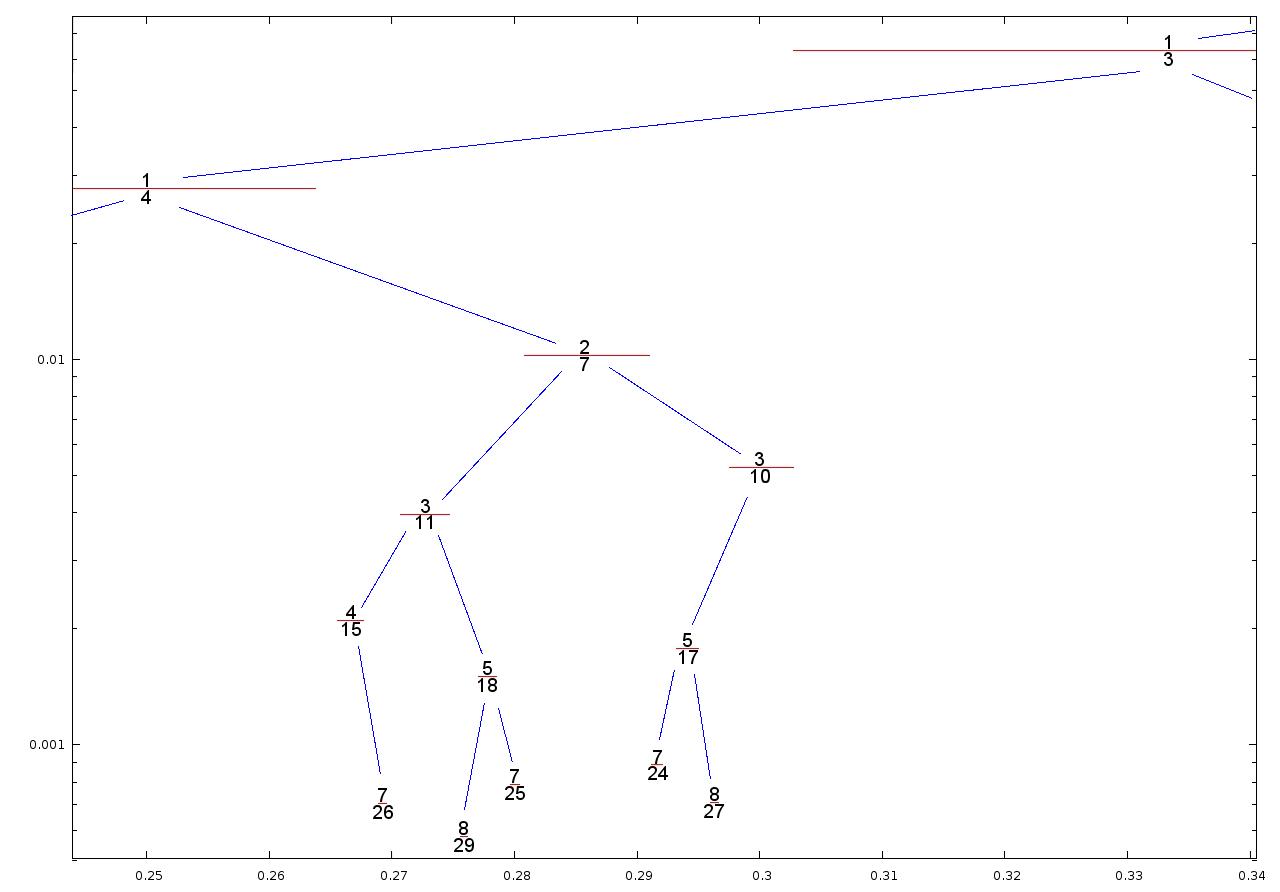}
  \caption{A few of the largest matching intervals in the window [1/4, 1/3], and the corresponding nodes of Stern-Brocot tree. The value on the y axis is the logarithm of the size of each interval.}
\label{demo}
\end{figure}

From our data it is also evident that the size of these intervals
 decreases as $k_1+k_2$ increases, and low order matchings tend to
disappear as $\alpha$ approaches zero. Moreover, as $\alpha$ tends to
$0$ the space covered by members of ``old'' families of type (ii)
encountered decreases, hence new families have to appear. One can
quantify this phenomenon from figure \ref{loglogscale}: since the size
of matching intervals in any family decreases as $1/n^3$ on the
interval cylinder $B_n$ (whose size decreases like $1/n^2$): this
means that, as n increases, the mass of $B_n$ gets more and more split
among a huge number of tiny intervals.

This fact compromises our numerical algorithm: it is clear that
choosing floating point values at random becomes a hopeless strategy
when approaching zero. Indeed, even if there still are intervals
bigger than the double-precision threshold, in most cases the random
seed will fall in a really tiny interval corresponding to a very high
matching order: this amounts to having very little gain as the result
of a really heavy computation.

We still can try to test numerically the conjecture that the matching set has full measure
on $[0,1]$; but we must expect that the percentage of space covered by matching intervals (found numerically)
will decrease dramatically near the origin, since we only detect intervals with $k_1+k_2$ bounded by some threshold. 
The matching intervals we have found so far cover a portion of $0.884$ of the interval $[0,1]$; this ratio increases to $0.989$ if we restrict to the interval
$[0.1,1]$ and it reaches $0.9989$ restricting to the interval 
$[0.2, 1]$.

\begin{figure}[h]
   \centering
   \includegraphics[scale=0.3]{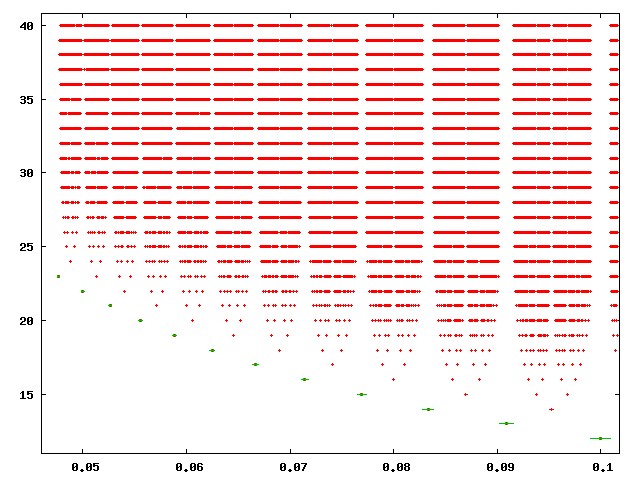}
   \caption{Dependence of the order 
     $k=k_1+k_2$ of a matching interval on the left endpoint  }
\label{hom}
  \end{figure}

The following graph represents the percentage of the interval $[x,1]$
which is covered by matching intervals of order $k=k_1+k_2$ for
different values of $k$\footnote{ Let us point out that for big values
  of $k$ the graph does not take into account all matching intervals
  of order $k$ but only those we have found so far. }.  It gives an
idea of the gain, in terms of the total size covered by matching
intervals, one gets when refining the gaps (i.e. considering matching intervals of higher order).

\begin{figure}[!h]   \centering
\includegraphics[scale=0.3]{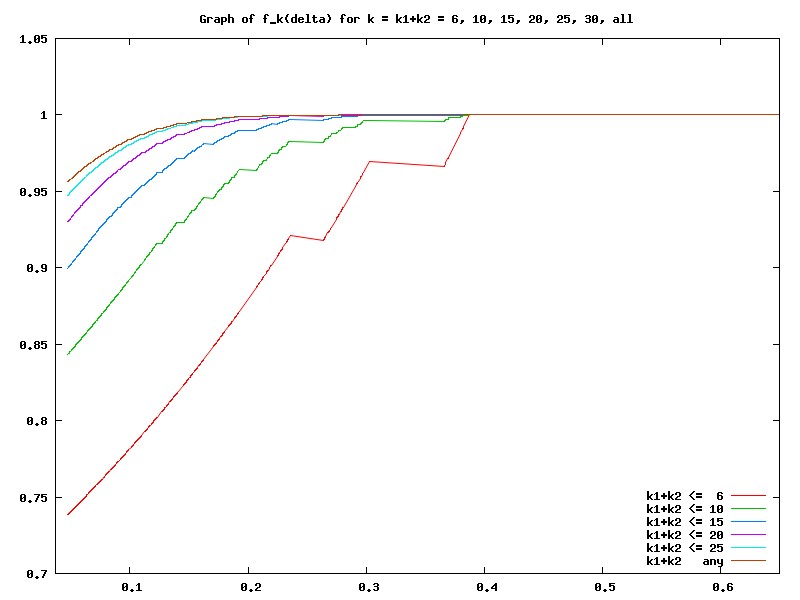}
    \caption{Percentage of covering by matching intervals}
\label{riempimento}
\end{figure}

Finally, to have a more precise idea of the relationship between order
of matching and size of the matching interval it is useful to see the
following scattered plot: the red dots correspond to matching
intervals found using a random seed, and the green ones to intervals found using the bisection algorithm. The two 
lines bounding the cloud correspond to matching intervals with very definite patterns: the upper line corresponds to the family $I_n$ (with endpoints of type $[0; \overline{n}]$ and $[0;\overline{ n-1, 1}]$), the lower line corresponds to matching intervals with endpoints of  type $[0; \overline{2,1,1,...,1,1,1}]$ and $[0; \overline{2,1,1,...,1,2}]$. The latter ones converge to 
$\frac{3-\sqrt{5}}{2}$, which is the supremum of all values where the entropy is increasing. 

Thus numerical evidence shows that, if $J$ is an interval with
matching order $k=k_1+k_2$ then the size of $J$ is bounded below by
$|J|\geq c_0 e^{-c_1k}$ where $c_0= 8.4423...$ and $c_1=
0.9624...$. On the other hand we know for sure that, on the right of
0.0475 (which corresponds to the leftmost matching interval of our
list), the biggest gap left by the matching intervals found so far is
of order $6.6\cdot 10^{-6}$.  So, if $J$ is a matching interval which
still does not belong to our list, either $J\subset [0,0.0475]$ and $k
\geq 20$ (see figure \ref{hom}), or its size must be smaller than $6.6\cdot 10^{-6}$ and by
the forementioned empirical rule, its order must be $k > 14.6$. Hence,
our list should include all matching intervals with $k_1 + k_2 \leq
14$.

\begin{figure}[!h]
   \centering
\includegraphics[scale=0.3]{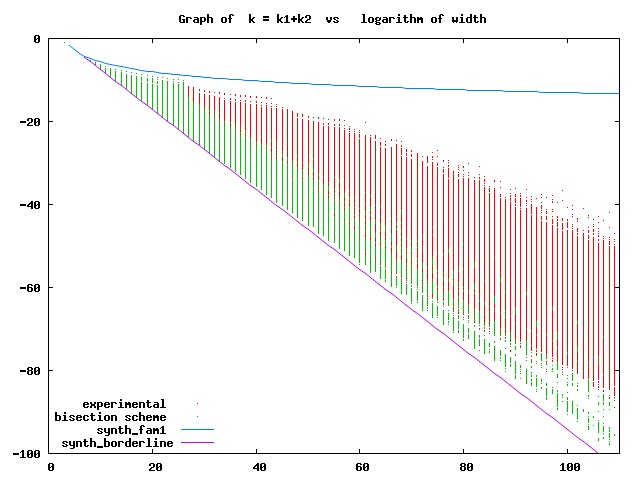}

  \caption{The order $k_1+k_2$ versus the logarithm of the size of the first $10^7$ matching intervals found.}
\label{logamp}
  \end{figure}

\subsection{The matching tree}\label{bintree}

As mentioned before, it seems that matching intervals are organized in a binary tree structure. 
To describe such structure, we will provide an algorithm which allows to construct all 
matching intervals by recursively ``filling the gaps'' between matching intervals previously obtained, similarly 
to the way the usual Cantor middle third set is constructed. 

In order to do so, let us first notice that every rational value $r \in \mathbb{Q}$ has two (standard) continued fraction expansions: 
$$r = [0; a_1,a_2, ... , a_m,  1]=[0; a_1,a_2, ... , a_m+ 1]$$
One can associate to $r$ the interval whose endpoints are the two quadratic surds with continued fraction obtained by endless repetition of the two expansions of $r$: 

\begin{definition} \label{conjstrings}
Given $r \in \mathbb{Q}$ with continued fraction expansion as above, we define $I_r$ to be the interval with endpoints  
$$[0;\overline{a_1,a_2, ... , a_m, 1}] \textup{ and } [0;\overline{a_1,a_2, ... , a_m+ 1}]$$
(in any order). The strings $S_1 := \{a_1, \dots, a_m, 1\}$ and $S_2 := \{a_1, \dots, a_m+1\}$ will be said to be
\emph{conjugate} and we will write $S_2 = (S_1)'$. 
\end{definition}

\noindent Notice that $r \in I_r$. It looks like all matching intervals are of type $I_r$ for some rational $r$. On the other hand, 

\begin{definition}
Given an open interval $I \supseteq [0,1]$ one can define the \emph{pseudocenter} of $I$ as the rational number $r \in I \cap \mathbb{Q}$ which has the minimum denominator among all rational numbers contained in $I$.
\end{definition}
\noindent It is straightforward to prove that the pseudocenter of an interval is unique, and the pseudocenter of $I_r$ is $r$ itself. 

We are now ready to describe the algorithm: 

\begin{enumerate}

\item The rightmost matching interval is $[\frac{\sqrt{5}-1}{2}, 1]$; its complement is the \emph{gap} 
$J = [0, \frac{\sqrt{5}-1}{2}]$.

\item Suppose we are given a finite set of intervals, called \emph{gaps of level $n$}, so that their complement is a union of matching intervals. Given each gap $J=[\alpha^-, \alpha^+]$, we determine its {\em
  pseudocenter} $r$. Let $\alpha^\pm=[0;S, a^\pm, S^\pm]$ be the continued fraction
  expansion of $\alpha^\pm$, where $S$ is the finite string containing
  the first common partial quotients, $a^+\neq a^-$ the first partial
  quotient on which the two values differ, and $S^\pm$ the rest of the
  expansion of $\alpha^\pm$, respectively. The pseudocenter of
  $[\alpha^-, \alpha^+]$ will be the rational number $r$ with expansions
  $[0;S, a, 1 ]=p/q=[0;S, a+1]$ where $a:= \min(a^+, a^-)$.

\item We remove from the gap $J$ the matching interval $I_r$ corresponding to the pseudocenter $r$: in this way the complement of $I_r$ in $J$ will consist of two intervals $J_1$ and $J_2$, which we will add to the list of gaps of level $n+1$. It might occur that one of these new intervals consists of only one point, i.e.  two  matching intervals  are  adjacent. 
\end{enumerate}

By iterating this procedure, after $n$ steps we will get a finite set $\mathcal{G}_n$ of gaps, and clearly $\bigcup_{J \in \mathcal{G}_{n+1}} J \subseteq \bigcup_{ J \in \mathcal{G}_n} J$. We conjecture all intervals obtained by taking pseudocenters of gaps are matching intervals, and that the set on which matching fails is the intersection 
$$\mathcal{G}_\infty:=\bigcap_{n\in \mathbb{N}} \bigcup_{J\in {\cal G}_n} J,$$ 

The next table contains the list of the elements of the family ${\cal G}_n$ of gaps of level $n$
for $n=0..4$: when a gap is reduced to a point we mark the corresponding line with the symbol $\star$.

$$
\begin{array}{|c|c l l|}
\hline
\vrule height1.2em width-1pt depth0.3em & \hbox to 1cm {\hfil\hfil} & \alpha^- & \alpha^+ \\[2pt]
\hline
\vrule height1.2em width-1pt depth0.3em{\cal G}_0 & & 0 & [0;\overline{1}\kern1.pt]\\[2pt]
\hline
\vrule height1.2em width-1pt depth0.3em{\cal G}_1 & & 0 & [0;\overline{2}\kern1.pt]\\[2pt]
           &\star &  [0;\overline{1, 1}\kern1.pt] & [0;\overline{1}\kern1.pt]\\[2pt]
 \hline
\vrule height1.2em width-1pt depth0.3em{\cal G}_2 & & 0 & [0;\overline{3}\kern1.pt]\\[2pt]
           & &  [0;\overline{2, 1}\kern1.pt] & [0;\overline{2}\kern1.pt]\\[2pt]
           &\star &  [0;\overline{1, 1}\kern1.pt] & [0;\overline{1}\kern1.pt]\\[2pt]
 \hline
\vrule height1.2em width-1pt depth0.3em{\cal G}_3 & & 0 & [0;\overline{4}\kern1.pt]\\[2pt]
           & &  [0;\overline{3, 1}\kern1.pt] & [0;\overline{3}\kern1.pt]\\[2pt]
           & &  [0;\overline{2, 1}\kern1.pt] & [0;\overline{2, 1, 1}\kern1.pt]\\[2pt]
           &\star &  [0;\overline{2, 2}\kern1.pt] & [0;\overline{2}\kern1.pt]\\[2pt]
           &\star &  [0;\overline{1, 1}\kern1.pt] & [0;\overline{1}\kern1.pt]\\[2pt]
\hline
\vrule height1.2em width-1pt depth0.3em{\cal G}_4 & & 0 & [0;\overline{5}\kern1.pt]\\[2pt]
           & &  [0;\overline{4, 1}\kern1.pt] & [0;\overline{4}\kern1.pt]\\[2pt]
           & &  [0;\overline{3, 1}\kern1.pt] & [0;\overline{3, 1, 1}\kern1.pt]\\[2pt]
           & &  [0;\overline{3, 2}\kern1.pt] & [0;\overline{3}\kern1.pt]\\[2pt]
           & &  [0;\overline{2, 1}\kern1.pt] & [0;\overline{2, 1,2}\kern1.pt]\\[2pt]
           & &  [0;\overline{2, 1, 1, 1 }\kern1.pt] 
	     & [0;\overline{2, 1, 1}\kern1.5pt]\\[2pt]
           &\star &  [0;\overline{2, 2}\kern1.pt] & [0;\overline{2}\kern1.pt]\\[2pt]
           &\star &  [0;\overline{1, 1}\kern1.pt] & [0;\overline{1}\kern1.pt]\\[2pt]
\hline
... & & ... & ... \\

\hline
\end{array}
$$

\begin{figure}[h]
   \centering
   \includegraphics[scale=0.4]{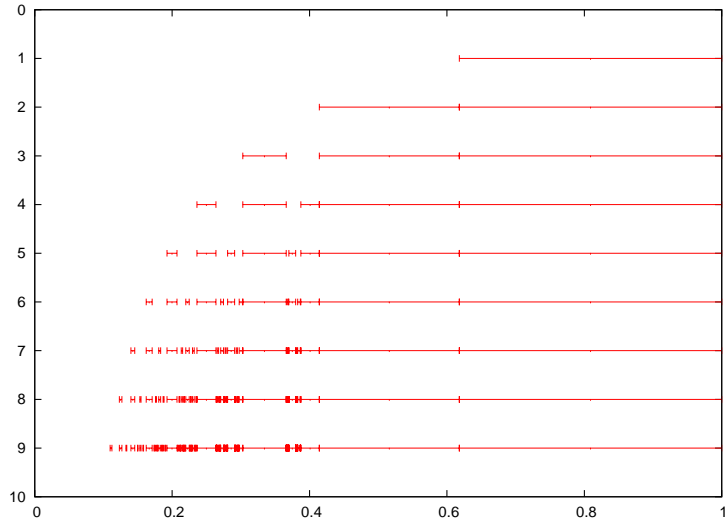}
   \caption{Recursive construction of the matching set}
\label{layers}
\end{figure}

We still cannot prove that this is the right scheme, but the numerical
evidence is quite robust: all $1.1 \cdot 10^6$ intervals obtained by running the first $23$ steps, for instance, turn out to be real matching intervals\footnote{We compared them to the list obtained as in section \ref{CAmatching}}. 

We can also prove the 

\begin{lemma} \label{lemmanbounded}
$\mathcal{G}_\infty$ consists of numbers of bounded type; more precisely, 
the elements of $\mathcal{G}_\infty \cap (\frac{1}{n+1}, \frac{1}{n}]$ have regular continued 
fraction bounded by $n$.
\end{lemma}

\begin{proof}
The scheme described before forces all endpoints of matching intervals containted in
the cylinder $B_n= ]1/(n+1), 1/n[$ to have quotients bounded by
    $n$. We now claim  that, if $\gamma=[0;c_1, c_2,..., c_n,
      ...]\notin {\cal M}$, then,  $c_k\leq c_1$ for
  all $k \in \mathbb{N}$. 

If  $\gamma\notin \mathcal{M}$ then
 $\gamma \in \bigcup_{J\in \mathcal{G}_n} J$ for all $n\in \mathbb{N}$; let us call $J_n$ the member
    of the family $\mathcal{G}_n$ containing $\gamma$. It may happen that there
  exists $n_0$ such that $J_n=\{\gamma\} \ \forall n\geq n_0$, so. $\gamma$ is an endpoint of two adiacent matching intervals, hence it has bounded type. Otherwise,
  $J_n=[\alpha_n, \beta_n]$ with $\beta_n-\alpha_n>0 \ \forall n>c_1$,
  where $\alpha_n, \beta_n$ are the endpoints of two matching
  intervals. Now, if   $p_n/q_n$ is the pseudocenter of $J_n$ from the minimality of $q_n$
  it follows that $|\beta_n -\alpha_n|< 2/q_n$, but also that
  $q_{n+1}>q_n$ (since $p_{n+1}/q_{n+1} \in J_{n+1}\subset J_n$); these
two properties together imply that $0 \leq \gamma - \alpha_n< 2/q_n \to 0$ as $n \to
  +\infty$. This implies $\gamma$ cannot be rational, since $\gamma \in
  J_n \ \forall n $ and the minimum denominator of a rational sitting
  in $J_n $ is $q_n \to +\infty$. Hence, since $\alpha_n \rightarrow \gamma$, 
  for every fixed $k \in \mathbb{N}$, there is some $n(k)$ such that for all
  $n\geq n(k)$ all the partial quotients up to level $k$ of $\gamma$
  coincide with those of $\alpha_n$, which are bounded by $c_1$.
\end{proof}

As a consequence, the validity of our algorithm ($\mathcal{G}_\infty = [0,1]\setminus \mathcal{M}$) would imply  
the conjecture \ref{nbounded}. \footnote{Our conjecture implies that also the set where the original conditions given by Nakada-Natsui hold has full measure;  the equivalent of lemma \ref{lemmanbounded} is, however, not true for their matching set, which differs from ours for a countable number of points.}

Notice $\mathcal{G}_{\infty} \cap (1/(n+1), 1/n]$ has Hausdorff dimension
strictly smaller than one for each $n$.  Moreover, the Hausdorff dimension of $n$-bounded numbers 
tends to $1$ as $n \rightarrow \infty$. We think that, similarly,  $\textup{H.dim}\{ (\frac{1}{n+1}, \frac{1}{n}] \setminus \mathcal{M} \} \rightarrow 1$: this would explain why finding
matching intervals near the origin becomes a tough task.

\begin{remark}
Since we have associated a rational number to each matching interval, one can think of the bisection 
algorithm as acting on $\mathbb{Q}$, and get a binary tree whose nodes are rationals: this object is related to
the well-known \emph{Stern-Brocot} tree. (For an introduction to it, see \cite{Knuth}). 
\end{remark}

Given that all matching intervals correspond to some rational number, one can ask which subset of $\mathbb{Q}$ actually arises in that way.

\begin{definition}
An interval $I_r$, $r \in \mathbb{Q}$ is \emph{maximal} if $I_r \supseteq I_{r'}$ $\forall r' \in I_r \cap \mathbb{Q}$.
\end{definition}

We conjecture that the matching intervals are \emph{precisely} the maximal intervals, so that the matching set is 
$$ \mathcal{M} = \bigcup_{r \in [0,1]\cap \mathbb{Q}} I_r = \bigcup_{\stackrel{r \in [0,1]\cap \mathbb{Q}}{I_r \textup{ maximal}}} I_r$$

As a matter of fact we can actually prove that the complement of the family $\mathcal{G}_n$ produced by the bisection algorithm consists of a family of maximal intervals: the proof of this fact is rather technical and  will appear in a forthcoming paper.

We have also found an empirical rule to reconstruct the periods $(k_1, k_2)$
 of a matching interval from  the labels of its enpoints. Let
$S=[a_1,...,a_\ell]$ be a label of the endpoint $s$ of some matching interval:
\begin{enumerate}
\item If $s$ is a left endpoint then 
$$ k_1=2+\sum_{j \ {\rm even}} a_j,\ \ \ k_2=\sum_{j \ {\rm odd}} a_j.$$ 
\item If $s$ is a right endpoint then 
$$ k_1=1+\sum_{j \ {\rm even}} a_j,\ \ \ k_2=1+\sum_{j \ {\rm odd}} a_j.$$ 
\end{enumerate}

Trusting this rule, we are able to prove that every neighbourhood of
the point $[0,3,\bar{1}]$ contains intervals of matching of all types:
with $k_1<k_2$, $k_1=k_2$ and $k_1>k_2$.  Indeed, it is not difficult
to realize that $[0,3,\bar{1}]$ is contained in the family of gaps
$J_P$ of endpoints $[0,\overline{3, P}]$ and $[0,\overline{3, P, 1}]$
where $P$ is a string of the type $1, 1, ..., 1, 1$ of even length; by
our rule the left endpoint of $J_P$ is the right endpoint of an
interval of matching where $k_1<k_2$. Nevertheless, performing a few
steps of the algorithm, it is not difficult to check that the gap
$J_P$ contains the interval $C_P$ of enpoints $[0,\overline{3, P, 2,
    1, 1}]$ and $[0,\overline{3, P, 2, 1, 1}]$ (on which $k_1=k_2$)
but also $D_P$ of enpoints $[0,\overline{3, P, 2, 1, 2, 1}]$and
$[0,\overline{3, P, 2, 1, 3}]$ (on which $k_1>k_2$).

\subsection{Adjacent intervals and period doubling} \label{1accum}

Let us now focus on pairs of adjacent intervals (corresponding to isolated points in $[0,1]\setminus\mathcal{M}$):
our data show they all come in infinite chains, and can be obtained from some starting matching interval via a ``period doubling'' construction.

\noindent
Let's start with a matching interval $\left]\alpha,\beta\right[\;$;
$\alpha= [0;\overline{S}\kern1pt]$ where $S$ is a sequence of positive
integer of odd length;
define the sequence of strings
\begin{equation} \label{recur}
\left\{
\begin{array}{l}
S_0=S\\
S_{n+1}=(S_nS_n)'
\end{array}
\right .
\end{equation}
where $S'$ denotes the conjugate of $S$ as in def. \ref{conjstrings}.
Let $a_n:=[0;\overline{S_{n}}]$  and $b_n:=[0;\overline{S'_n}]$; then
the sequence $I_n:=]a_n,b_n[$ is formed by a chain of adjacent
    intervals: clearly $b_{n+1}=a_n$, moreover $a_n<b_n$ because
    $|S_n|$ is odd for all $n$.

Assuming this scheme, we can construct many cluster points of matching intervals. For instance, let us look at the first (i.e. rightmost) one: we start with the interval $](\sqrt{5}-1)/2, \ 1[$ so that the first terms of the sequence $S_n$ are

$$
\begin{array}{lll}
S_0 &=&(1)\\
S_1 &=&(2)\\
S_3 &=&(2,1,1)\\
S_4 &=&(2,1,1,2,2)\\
S_5 &=&(2,1,1,2,2,2,1,1,2,1,1)
\end{array}
$$

The corresponding sequence
$a_n$ converges to the first (i.e. rightmost) point $\hat{\alpha}$ where intervals of
matching cluster. We can also  determine the continued fraction expansion of
 the value $\hat{\alpha}$, since it can be obtained just merging\footnote{This can be done since, by \eqref{recur}, $S_n$ is a substring of $S_{n+1}$.} the strings $(S_n)_n\in \mathbb{N}$
\begin{footnotesize}
$$\hat{\alpha}=[0,2,1,1,2,2,2,1,1,2,1,1,2,1,1,2,2,2,1,1,2,2,2,1,1,2,2,2,1,1,2,1,1,2,1,1,2,2,2,...]$$
\end{footnotesize}
Numerically\footnote{This pattern has been checked up to level 10,
  which corresponds to a matching interval of size smaller than
  $10^{-200}$; see also the second table
  in section \ref{tables}.},
 $\hat{\alpha} \cong
0.386749970714300706171524803485580939661$\dots

It is evident from formula \eqref{recur} that any such cluster point will be a bounded-type number;
one can indeed prove that no cluster point of this type is a quadratic surd.

\section{Behaviour of entropy inside the matching set}

In \cite{NakadaNatsui}, the following formula is used to relate the change of entropy between two sufficiently
close values of $\alpha$ to the invariant measure corresponding to one
of these values: more precisely

\begin{proposition} \label{fmlaNakada}
Let us suppose the hypotheses of prop. \ref{matchNakada} hold for $\alpha$: then for $\eta > 0$ small enough 

\begin{equation} \label{derivEntropy}
h(T_{\alpha-\eta}) = \frac{h(T_{\alpha})}{1 + (k_2 - k_1)\mu_{\alpha}([\alpha-\eta, \alpha])}
\end{equation}
and similarly
\begin{equation} \label{derivEntropy2}
h(T_{\alpha}) = \frac{h(T_{\alpha+\eta})}{1 + (k_2 - k_1)\mu_{\alpha+\eta}([\alpha, \alpha+\eta])}
\end{equation}
\end{proposition}

By exploiting these formulas, we will get some results on the behaviour of $h(T_{\alpha})$.

\subsection{One-sided differentiability of $h(T_\alpha)$}

Equation (\ref{derivEntropy}) has interesting consequences on the differentability of $h$: 
we can rewrite it as
$$h(T_{\alpha}) - h(T_{\alpha-\eta}) = h(T_{\alpha-\eta})(k_2 - k_1)\mu_{\alpha}([\alpha-\eta, \alpha])$$
and dividing by $\eta$
$$\frac{h(T_{\alpha}) - h(T_{\alpha-\eta})}{\eta} = h(T_{\alpha-\eta})(k_2 - k_1)\frac{\mu_{\alpha}([\alpha-\eta, \alpha])}{\eta}$$
\\
Since $\rho_{\alpha}$ has bounded variation, then there exists $R(\alpha) = \lim_{x\rightarrow \alpha^{-}} \rho_{\alpha}(x)$, therefore  
 $$\lim_{\eta\rightarrow 0} \frac{\mu_{\alpha}([\alpha-\eta, \alpha])}{\eta}= R(\alpha)$$
  and by the continuity of $h$ (which is obvious in this case by equation (\ref{derivEntropy}))
$$\lim_{\eta \rightarrow 0}\frac{h(T_{\alpha})-h(T_{\alpha-\eta})}{\eta} = h(T_{\alpha})(k_2 - k_1)\lim_{x\rightarrow \alpha^{-}} \rho_{\alpha}(x)$$
hence the function $\alpha \mapsto h(T_{\alpha})$ is left differentiable in $\alpha$. On the other hand, one can slightly modify the proof of $(\ref{derivEntropy2})$ and realize it is equivalent to

$$h(T_{\alpha+\eta}) =\frac{h(T_{\alpha})}{1 + (k_1-k_2)\mu_{\alpha}([\alpha-1, \alpha-1 + \eta])}$$ 
which reduces to
$$\frac{h(T_{\alpha+\eta}) - h(T_{\alpha})}{\eta} = \frac{ \mu_{\alpha}([\alpha-1, \alpha-1+\eta])}{\eta} \frac{ h(T_{\alpha})(k_2-k_1)}{1 + (k_1 - k_2)\mu_{\alpha}([\alpha-1, \alpha-1+\eta])}$$

\noindent
Since the limit
$$\lim_{\eta \rightarrow 0} \frac{\mu_{\alpha}([\alpha-1, \alpha-1+\eta])}{\eta} = \lim_{x \rightarrow (\alpha-1)^{+}} \rho_{\alpha}(x)$$ 
also exists, then $h(T_{\alpha})$ is also right differentiable in $\alpha$, more precisely
$$\lim_{ \eta \rightarrow 0} \frac{h(T_{\alpha+\eta}) - h(T_{\alpha})}{\eta} = h(T_{\alpha})(k_2-k_1) \lim_{x \rightarrow (\alpha-1)^{+}} \rho_{\alpha}(x)$$

We conjecture that in such points the left and right derivatives are equal. This is trivial for $k_1 = k_2$; for $k_1 \neq k_2$ it is equivalent to say $\lim_{x\rightarrow \alpha^{-}} \rho_{\alpha}(x) = \lim_{x \rightarrow (\alpha-1)^{+}} \rho_{\alpha}(x)$.

\subsection{The entropy for $\alpha \geq \frac{2}{5}$}

\begin{corollary} 
For $\frac{2}{5} \leq \alpha \leq \sqrt{2}-1$, the entropy is 
$$h(T_{\alpha}) = \frac{\pi^2}{6 \log \left(\frac{\sqrt{5}+1}{2}\right)}$$
\end{corollary}
\begin{proof}
Every $\alpha$ in the interval $(0.4 ,\sqrt{2}-1)$ satisfies the hypotheses of the theorem with $k_1 = k_2 = 3$, hence $h(T_{\alpha})$ is locally constant, and by continuity $h(T_{\alpha}) = h(T_{\sqrt{2}-1})$, whose value was already known.

\end{proof}

\begin{remark}
By using our computer-generated matching intervals, we can analogously prove $h(T_{\alpha}) = h(T_{\sqrt{2}-1})$ for \begin{small}$\sqrt{2}-1 \geq \alpha \geq  0.386749970714300706171524...$\end{small}
\end{remark}

\subsection{Invariant densities}

In the case $\alpha \geq \sqrt{2}-1$ it is known that invariant densities are of the form
$$\rho_{\alpha}(x) = \sum_{i = 1}^r \chi_{I_i}(x)\frac{A_i}{x+B_i}$$
where the $I_i$ are subintervals of $[\alpha-1, \alpha]$. 

For these values of $\alpha$, a matching condition is present and the endpoints of the $I_i$ (i.e. the values where the density may ``jump'') correspond exactly to the first few iterates of $\alpha$ and $\alpha-1$ under the action of $T_{\alpha}$. We present some numerical evidence in order to support the 

\begin{conjecture}\label{NN}
Let $\alpha \in [0,1]$ be a value such that one has a matching of type $(k_1, k_2)$ (i.e. with $T_{\alpha}^{k_1}(\alpha) = T_{\alpha}^{k_2}(\alpha-1)$). Then the invariant density has the form
\begin{equation}\label{density}\rho_{\alpha}(x) = \sum_{i = 1}^r \chi_{I_i}(x)\frac{A_i}{x+B_i} \end{equation}
where each $I_i$ is an interval with endpoints contained in the set $$S := \{T_{\alpha}^m(\alpha) : 0 \leq m < k_1 \} \cup \{T_{\alpha}^n(\alpha-1) :  0 \leq n < k_2 \}$$
Therefore, the number of branches is bounded above by $k_1+k_2-1$.
\end{conjecture}

In all known cases, moreover, there exists exactly one $I_i$ which contains $\alpha$ and exactly one which contains $\alpha-1$; thus, on neighbourhoods of $\alpha$ and $\alpha-1$, the invariant density has the simple form
$\rho_{\alpha}\vert_{I_i}(x) = \frac{A_i}{x + B_i}$

As an example of such numerical evidence we report a numerical simulation of
the invariant density for some values of $\alpha$ in the interval
$[\frac{\sqrt{13}-3}{2},\frac{\sqrt{3}-1}{2}]$ where a matching of
type $(2,3)$ occurs. We fit the invariant density with the function $A_+/(x+B_+)$ on the interval $[\max\{S\}, \alpha]$ and with the function $A_-/(x+B_-)$ on $[\alpha-1, \min\{S\}]$.

$$
\begin{array}{|l|l|l|l|l|l|l|}
\hline 
& \alpha = 0.310 & \alpha = 0.320 & \alpha = \frac{1}{3} & \alpha = 0.338 & \alpha = 0.350 & \alpha = 0.360 \\
\hline
A_+ & 1.76114 & 1.76525 & 1.77603 & 1.78963 & 1.81981 & 1.84658 \\
B_+ & 1.64768 & 1.63487 & 1.62374 & 1.62987 & 1.64092 & 1.65138 \\ \hline
A_- & 1.77289 & 1.78874 & 1.81488 & 1.82411 & 1.84562 & 1.85959 \\ 
B_- & 2.66097 & 2.66081 & 2.66583 & 2.66751 & 2.66915 & 2.6658 \\ \hline
\end{array}
$$

Moreover, from these numerical data it is apparent that the leftmost branch of hyperbola is nothing else that a translation by $1$ of the rightmost one (i.e. $A_+ = A_-$, $B_- = B_++1$).

\begin{figure}[h]
   \centering
   \includegraphics[scale=0.28]{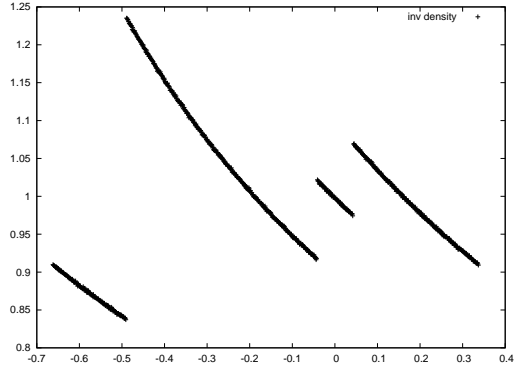}
   \caption{Invariant density for $\alpha = 0.338$}
\label{invdensity338}
\end{figure}

\subsection{Comparison with the entropy}

If $I \subset [0,1]$ is a matching interval, the knowledge of the invariant density for one single value of $\alpha \in I$ plus eq. (\ref{derivEntropy}) allows us to recover the entropy in the whole interval. Let $\alpha$ belong to an interval where a matching of type $(k_1, k_2)$ occurs and suppose, according to the previous conjecture, that on $[\overline{x}, \alpha]$ the invariant density has the form
$$\rho_{\alpha}(x) = \frac{A}{x+B}$$ for some $A, B \in \mathbb{R}$ and $\overline{x} = \max \{ T_{\alpha}^n(\alpha), 1 \leq n < k_1 \} \cup \{ T_{\alpha}^{m}(\alpha-1), 1 \leq m < k_2 \}$.
Then by (\ref{derivEntropy}), for $x < \alpha$ sufficiently close to $\alpha$

\begin{equation} \label{log} h(x) = \frac{h(\alpha)}{1+(k_2-k_1)A\log\left(\frac{B+\alpha}{B+x}\right)} \end{equation}
We think that the entropy has in general such form for values of $\alpha$ where a matching occurs.

Let us consider the particular case of the interval $[0.295 , 0.3042]$. In the region to the right of the big central plateau (i.e. for $\alpha > \frac{-3+\sqrt{13}}{2}$) the behaviour of entropy looks approximately linearly increasing, as conjectured in \cite{LuzziMarmi}, sect. 3. We will provide numerical evidence it actually has the logarithmic form given by equation \eqref{log} on the interval $[\frac{\sqrt{13}-3}{2},\frac{\sqrt{3}-1}{2}]$. To test this hypothesis, we proceed as follows:

\begin{enumerate}
\item We fit the data of the invariant density for $\alpha = 0.338$,
  obtaining the constants $A_+$ and $B_+$ which refer to the rightmost
  branch of hyperbola (the data are already in the previous table).

 \item We fit the data of the entropy already calculated (relative to
   the window $ [0.30277,0.3042]$) with the function $\eqref{log}$. We
   assume $A_+$ and $B_+$ as given constants and we look for the best
   possible value of $h(\alpha)$ (which we did not have from previous
   computations). The result given is $h(\alpha) \cong 3.28311$. In
   the figure we plot the obtained function in the known
   window, as well as a linear fit. In this interval, the difference between the two functions is negligible.
 (Figure \ref{outputsinistro})

\item In order to really distinguish between linear and logarithmic
  behaviour of the entropy, we computed some more numerical data for
  the entropy far away to the right but in the same matching
  interval. In this region the linear and logarithmic plots are
  clearly distinguishable, and the new points seem to perfectly agree
  with the logarithmic formula\footnote{Let us remark that the new
    values computed are just a few, but are more accurate than those
    in the interval $ [0.30277,0.3042]$ since we used the package CLN
    a C++ library to perform computations in arbitrary
    precision}. (Figure \ref{outputdestro}) 

\begin{figure}[htbp]
   \begin{minipage}{0.5\textwidth}
    \centering
    \includegraphics[scale=0.23]{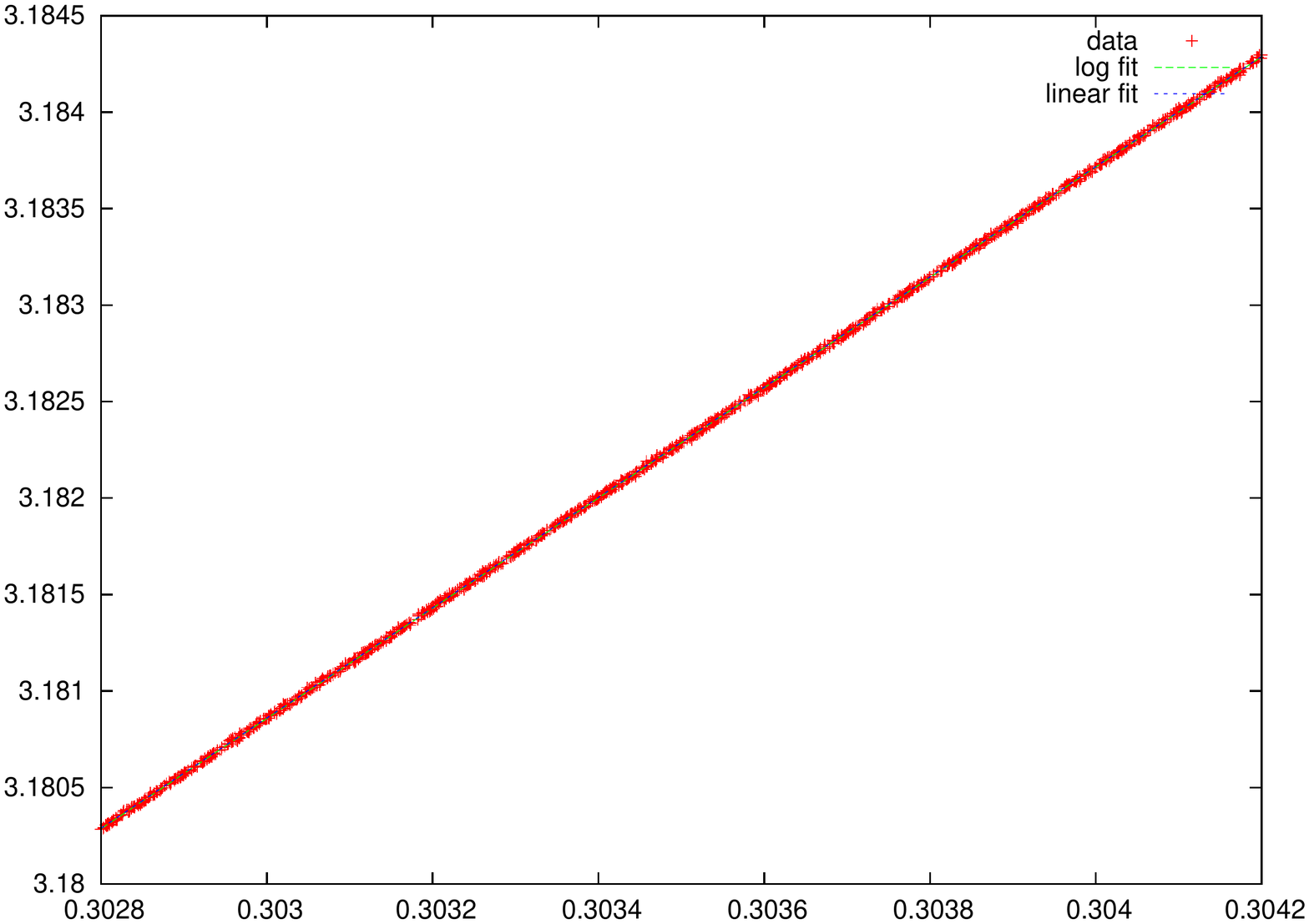}
   \caption{linear vs logarithmic fit, $0.3028 \leq \alpha \leq 0.3042$}
\label{outputsinistro} \end{minipage}
\begin{minipage}{0.5\textwidth}
     \includegraphics[scale=0.23]{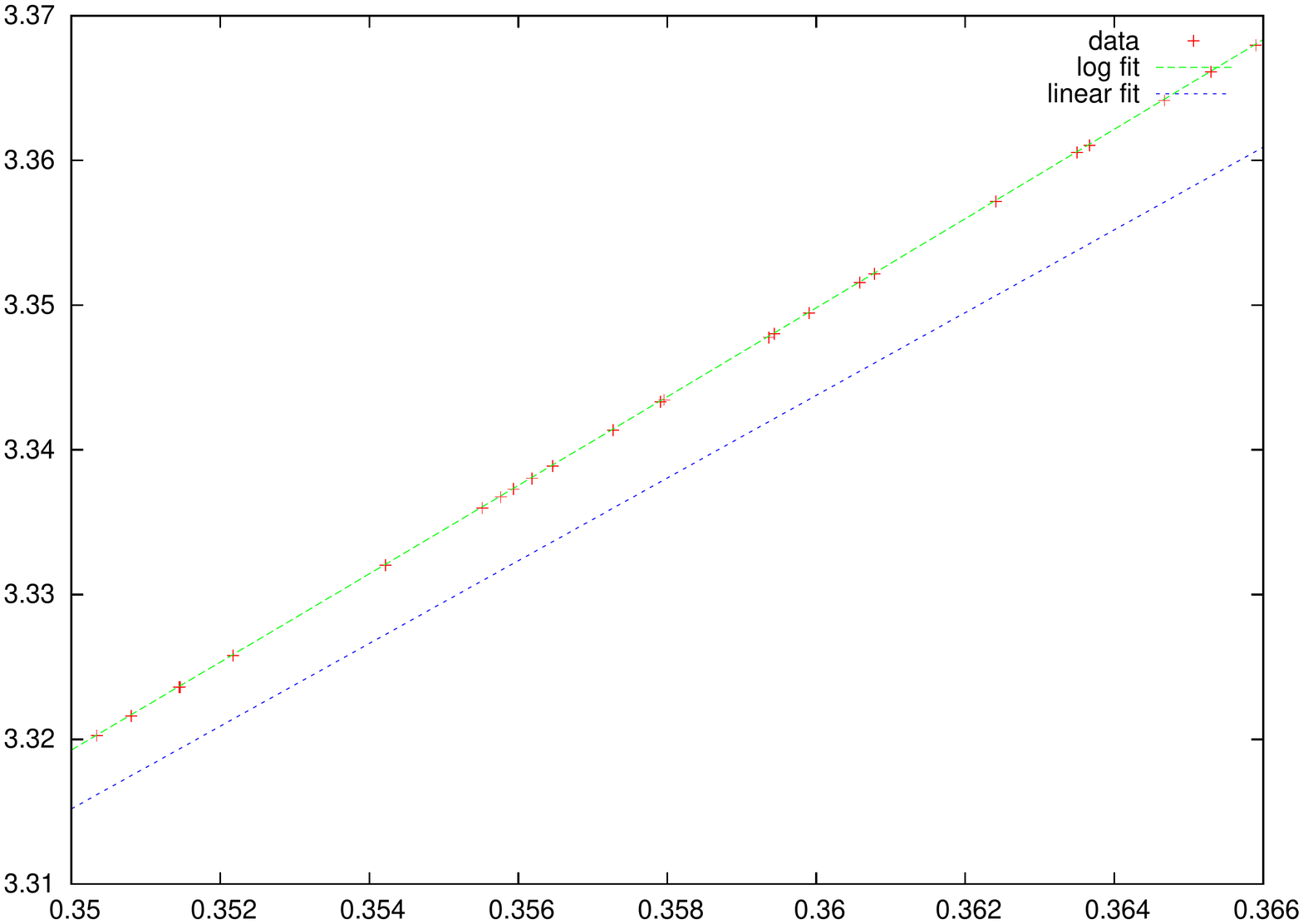}
  \caption{linear vs logarithmic fit, $0.35 \leq \alpha \leq 0.366$}
\label{outputdestro}   
\end{minipage}
\end{figure}

Notice these data agree
  with eq. \ref{log} also for $x > \alpha$, which is equivalent to say
  $\rho_\alpha(x) = \frac{A}{B+1+x}$ for $x$ in a right neighbourhood
  of $\alpha-1$.

\end{enumerate}

\section{Appendix}
In this appendix we give the proof of two simple results which are
of some relevance for the issues discussed in this paper.

\begin{proposition}
If $x_0$ is a quadratic surd then  $x_0$ is a preperiodic point for $T_\alpha$, $\alpha\in [0,1]$.
\end{proposition}
For $\alpha=1$ this is the well known Lagrange Theorem, and this statement is known to be true for
 $\alpha=0$ and $\alpha\in [1/2,1]$ \cite{HartonoKraaikamp}. Since we did not find a reference containing a simple proof
of this fact for all $\alpha \in [0,1]$ we sketch it here, in few lines: this proof follows closely
 the classical  proof of Lagrange Theorem for regular continued fractions 
given by  \cite{Charves} which relies on approximation 
properties of convergents, therefore it works for $\alpha>0$. 

If $x_0$ is a quadratic surd then $F_0(x_0)=0$ for some  $F_0(x):=A_0x^2+B_0x+C_0$
quadratic polynomial with integer coefficients. 
On the other hand, since\footnote{To simplify notations we shall write $p_n, \ q_n$ instead of $p_{n,\alpha}, \ q_{n, \alpha}$. }
$x_0=\frac{p_{n-1}x_{n}+p_n}{q_{n-1}x_n+q_n}$,
setting
$F_n(x):=F_0(\frac{p_{n-1}x+p_n}{q_{n-1}x+q_n})(q_{n-1}x+q_n)^2$,
we get that $F_n(x_n)=F_0(x_0)=0$. 

Moreover $F_n(x)=A_nx^2+B_nx+C_n$
with
\begin{equation}\label{chiave}
A_n=F_0(p_{n-1}/q_{n-1})q_{n-1}^2,  \ \ C_n=F_0(p_{n}/q_{n})q_{n}^2, \ \  B_n^2-4A_nC_n=B_0^2-4A_0C_0.
\end{equation}

Both $A_n, \ B_n$ are bounded since:
$|F_0(p_n/q_n)|=|F_0(p_n/q_n)-F_0(x_0)|=|F_0'(\xi)| |\frac{p_n}{q_n}-x_0|\leq
\frac{C}{\alpha q_n^2}$; moreover from the last equation in \eqref{chiave} it follows that
$B_n$ are bounded as well.

\begin{proposition}
The variance $\sigma^2(\alpha)$ is constant for  $\alpha \in [\sqrt{2}-1, (\sqrt{5}-1)/2]$.  
\end{proposition}

This result relies on the fact that for all $\alpha\in
[\sqrt{2}-1, (\sqrt{5}-1)/2]$ the maps  $T_\alpha$  have natural extensions
$\tilde T_\alpha$ which are all isomorphic to $\tilde T_{1/2}$.  In
the following we shall prove the claim for $\alpha \in [\sqrt{2}-1, 1/2]$
  and we shall write $T_1$ instead of $T_\alpha$ and $T_2$ instead of
  $T_{1/2}$.  So $T_j: I_j\to I_j, \ (j=1,2)$ are 1-dimensional map
  with invariant measure $\mu_j$; $\tilde T_j: \tilde I_j\to \tilde
  I_j, \ (j=1,2)$ are the corresponding 2-dimensional representations
  of the natural extension with invariant measure $\tilde\mu_j$, and
  $\Phi: \tilde I_1\to \tilde I_2$ is the (measurable) isomorphism
$$ \Phi \circ \tilde{T}_1 = \tilde{T}_2 \circ \Phi, \ \ \Phi_{*}\tilde\mu_1=\tilde\mu_2 $$
First let us point out (see \cite{NakadaNatsui} pg 1222-1223) that 
$\Phi$ is almost everywhere differentiable and has a diagonal differential; 
moreover $\tilde T_j$ are almost everywhere differentiable as well and have triangular differential.
Therefore
\begin{equation}
 d\Phi\vert_{T_1(x,y)} d\tilde{T}_1 \vert_{(x,y)} = d\tilde{T}_2\vert_{\Phi(x,y)} d\Phi_{(x,y)} 
\end{equation}
and it ie easy to check that, setting $\tilde T_j^x$ the first component of $\tilde T_j$, a scalar analogue  holds as well
\begin{equation}
\frac{\partial \Phi^x}{\partial x}\vert_{T_1(x,y)} \frac{\partial \tilde{T}_1^x}{\partial x}\vert_{(x,y)} = \frac{\partial \tilde{T}_2^x}{\partial x}\vert_{\Phi(x,y)} \frac{\partial \Phi^x}{\partial x}\vert_{(x,y)}
\end{equation}
So we get that, for all $k$,

$$ \log \left| \frac{\partial \tilde{T}_1^x}{\partial x} \right| =
\log \left| \frac{\partial \tilde{T}_2^x}{\partial x} \circ \Phi
\right| + \log \left| \frac{\partial \Phi^x}{\partial x}  \right| - \log \left| \frac{\partial \Phi^x}{\partial x}
\circ \tilde{T}_1\right| $$ 
Since $\tilde T_1^x$ is $\tilde \mu_1$-measure preserving $\int_{\tilde I_1} \log \left|
\frac{\partial \Phi^x}{\partial x} \right| - \log \left|
\frac{\partial \Phi^x}{\partial x} \circ \tilde{T}_1\right| d \tilde \mu_1 =0$; so,

taking into account that $\Phi_{\tilde{\mu_1}}=\tilde\mu_2 $ we get

$$\int_{\tilde I_1} \log \left| \frac{\partial \tilde{T}_1^x}{\partial
  x} \right| d \tilde \mu_1 = \int_{\tilde I_1} \log \left|
\frac{\partial \tilde{T}_2^x}{\partial x} \circ \Phi \right| d \tilde
\mu_1 = \int_{\tilde I_2} \log \left| \frac{\partial
  \tilde{T}_2^x}{\partial x} \right| d \tilde \mu_2 :=m.
$$

Let us define 
$g_1:= \log \left| \frac{\partial
  \tilde{T}_1^x}{\partial x} \right|$ 
and 
$g_2:= \log \left|
\frac{\partial \tilde{T}_2^x}{\partial x} \right|$ 
(so that
$\int_{\tilde I_1} g_1 d \tilde \mu_1=\int_{\tilde I_2} g_2 d \tilde
\mu_2=0$) and $S_N^Tg:=\sum_{k=0}^{N-1} g \circ T^k$; we easily see that
$$S_N^{\tilde{T}_1} g_1 = S_N^{\tilde{T}_2} g_1 \circ \Phi
\log \left| \frac{\partial \Phi^x}{\partial x} \circ \tilde{T}_1^k \right| - \log \left| \frac{\partial \Phi^x}{\partial x} \circ \tilde{T}_1^{k+1}\right| $$
which means that 
$S_N^{\tilde{T}_1} g_1$ and $S_N^{\tilde{T}_2} g_2 \circ \Phi$ differ by a coboundary.

\begin{lemma}
Let $u,v$ be two observables such that
\begin{enumerate}
\item $\lim_{N\to +\infty}\int(\frac{S_N v}{\sqrt N})^2d\mu=l \in \mathbb{R}$;
\item $u=v+ (f- f\circ T)$ for some $f \in L^2$. 
\end{enumerate} 
Then  
 $$\lim_{N\to +\infty}\int(\frac{S_N v}{\sqrt 
N})^2d\mu= \lim_{N\to +\infty}\int(\frac{S_N u}{\sqrt N})^2d\mu.$$
\end{lemma}
The lemma implies  
\begin{small}\begin{equation}\label{limit}
\lim_{N\to +\infty}\int_{\tilde I_1}  \left(\frac{ S_N^{\tilde{T}_1} g_1}{\sqrt N}\right)^2 d \tilde \mu_1=\lim_{N\to +\infty}\int_{\tilde I_2}\left(\frac{ S_N^{\tilde{T}_2} g_2 }{\sqrt N}\right)^2 d \tilde \mu_2
\end{equation}\end{small}
This information can be translated back to the original systems:
since  $\frac{\partial\tilde{T}_1^x}{\partial x}\vert_{(x,y)} =
T_1'(x)$,  $\frac{\partial\tilde{T}_2^x}{\partial x}\vert_{(x,y)} =
T_2'(x)$ if we define
$$G_1 := \log |T_1'(x)| - \int_{I_1} \log |T_1'(x)| d\mu_1$$ 
 $$G_2 = \log |T_2'(x)| - \int_{I_2} \log |T_2'(x)| d\mu_2$$ 
we get
 $g_1(x,y) = G_1(x)$ and $g_2(x,y) = G_2(x)$; therefore $S_N^{\tilde{T}_1}g_1 = S_N^{T_1}G_1$ and $S_N^{\tilde{T}_2}g_2 = S_N^{T_2}G_2$.
Finally, by equation \eqref{limit}, we get
$$
\lim_{N\to +\infty}
\int_{I_1} \left( \frac{S_N^{T_1}G_1}{\sqrt{N}} \right)^2 d\mu_1 = 
\lim_{N\to +\infty}
\int_{I_2} \left( \frac{S_N^{T_2}G_2}{\sqrt{N}} \right)^2 d\mu_2$$

\subsection{Tables}\label{tables}

\def\SomeSpc{5pt}
\begin{flushleft}
\scalebox{0.80}{$\begin{array}{|llll|llll|}
\hline
\VSpc{1.4em}{0.3em}
(k_1\; k_2) & \textup{size} & \left(\alpha^-\right., & \left.\alpha^+\right) & (k_1\; k_2) & \textup{size} & \left(\alpha^-\right., & \left.\alpha^+\right)  \\[\SomeSpc]
\hline
\VSpc{1.4em}{0.3em}
(3\;9) & \hbox{7.69e-4} & \left( \frac{-8 + \sqrt{82}}{9}\right.\,,\,&\left.\frac{-2 + \sqrt{5}}{2}\right) &
(8\;6) & \hbox{6.42e-5} & \left( \frac{-33 + \sqrt{2305}}{64}\right.\,,\,&\left.\frac{-77 + \sqrt{7221}}{34}\right) \\[\SomeSpc]
(2\;8)
& \hbox{3.68e-3} & \left( -4 + \sqrt{17} \VSpc{1.2em}{0.3em}\right.\,,\,&\left.\frac{-7 + \sqrt{77}}{14}\right) &
(5\;5)
& \hbox{1.46e-3} & \left( \frac{-7 + \sqrt{101}}{13}\right.\,,\,&\left.\VSpc{1.2em}{0.3em} -2 + \sqrt{5}\right) \\[\SomeSpc]
(3\;8) & \hbox{1.11e-3} & \left( \frac{-7 + \sqrt{65}}{8}\right.\,,\,&\left.\frac{-7 + 3\,\sqrt{7}}{7}\right) &
(2\;4) & \hbox{2.77e-2} & \left( \VSpc{1.2em}{0.3em} -2 + \sqrt{5}\right.\,,\,&\left.\frac{-3 + \sqrt{21}}{6}\right) \\[\SomeSpc]
(2\;7) & \hbox{5.44e-3} & \left( \frac{-7 + \sqrt{53}}{2}\right.\,,\,&\left.\frac{-3 + \sqrt{15}}{6}\right) &
(3\;6) & \hbox{2.1e-3} & \left( \frac{-7 + \sqrt{65}}{4}\right.\,,\,&\left.\frac{-6 + 4\,\sqrt{5}}{11}\right) \\[\SomeSpc]
(3\;8) & \hbox{6.98e-4} & \left( \frac{-19 + \sqrt{445}}{14}\right.\,,\,&\left.\frac{-9 + 2\,\sqrt{30}}{13}\right) &
(4\;6) & \hbox{7.02e-4} & \left( \frac{-11 + \sqrt{226}}{15}\right.\,,\,&\left.\frac{-23 + 3\,\sqrt{93}}{22}\right) \\[\SomeSpc]
(3\;7) & \hbox{1.69e-3} & \left( \frac{-6 + 5\,\sqrt{2}}{7}\right.\,,\,&\left.\frac{-3 + 2\,\sqrt{3}}{3}\right) &
(3\;5) & \hbox{3.97e-3} & \left( \frac{-5 + \sqrt{37}}{4}\right.\,,\,&\left.\frac{-9 + \sqrt{165}}{14}\right) \\[\SomeSpc]
(4\;7) & \hbox{8.12e-4} & \left( \frac{-17 + \sqrt{445}}{26}\right.\,,\,&\left.\frac{-3 + \sqrt{11}}{2}\right) &
(4\;6) & \hbox{5.77e-4} & \left( \frac{-13 + \sqrt{257}}{11}\right.\,,\,&\left.\frac{-2 + 2\,\sqrt{2}}{3}\right) \\[\SomeSpc]
(2\;6) & \hbox{8.54e-3} & \left( \VSpc{1.2em}{0.3em}-3 + \sqrt{10}\right.\,,\,&\left.\frac{-5 + 3\,\sqrt{5}}{10}\right) &
(4\;5) & \hbox{1.51e-3} & \left( \frac{-15 + \sqrt{445}}{22}\right.\,,\,&\left.\frac{-8 + 3\,\sqrt{11}}{7}\right) \\[\SomeSpc]
(3\;8) & \hbox{6.06e-4} & \left( \frac{-11 + \sqrt{145}}{6}\right.\,,\,&\left.\frac{-10 + 2\,\sqrt{42}}{17}\right) &
(5\;5) & \hbox{7.88e-4} & \left( \frac{-10 + \sqrt{226}}{18}\right.\,,\,&\left.\frac{-23 + 5\,\sqrt{29}}{14}\right) \\[\SomeSpc]
(3\;7) & \hbox{1.12e-3} & \left( \frac{-8 + \sqrt{82}}{6}\right.\,,\,&\left.\frac{-15 + \sqrt{357}}{22}\right) &
(3\;4) & \hbox{1.02e-2} & \left( \frac{-3 + \sqrt{17}}{4}\right.\,,\,&\left.\frac{-3 + \sqrt{15}}{3}\right) \\[\SomeSpc]
(3\;6) & \hbox{2.76e-3} & \left( \frac{-5 + \sqrt{37}}{6}\right.\,,\,&\left.\frac{-5 + \sqrt{35}}{5}\right) &
(4\;6) & \hbox{8.86e-4} & \left( \frac{-11 + \sqrt{170}}{7}\right.\,,\,&\left.\frac{-19 + 3\,\sqrt{93}}{34}\right) \\[\SomeSpc]
(4\;6) & \hbox{1.34e-3} & \left( \frac{-7 + \sqrt{82}}{11}\right.\,,\,&\left.\frac{-15 + \sqrt{285}}{10}\right) &
(4\;5) & \hbox{1.78e-3} & \left( \frac{-15 + \sqrt{365}}{14}\right.\,,\,&\left.\frac{-7 + 3\,\sqrt{11}}{10}\right) \\[\SomeSpc]
(9\;7) & \hbox{2.38e-5} & \left( \frac{-51 + 13\,\sqrt{29}}{100}\right.\,,\,&\left.\frac{-117 + \sqrt{15621}}{42}\right) &
(5\;5) & \hbox{7.09e-4} & \left( \frac{-11 + \sqrt{257}}{17}\right.\,,\,&\left.\frac{-6 + 2\,\sqrt{14}}{5}\right) \\[\SomeSpc]
(5\;6) & \hbox{7.91e-4} & \left( \frac{-9 + \sqrt{145}}{16}\right.\,,\,&\left.\frac{-10 + 2\,\sqrt{30}}{5}\right) &
(8\;6) & \hbox{2.73e-5} & \left( \frac{-54 + \sqrt{7057}}{101}\right.\,,\,&\left.\frac{-127 + 7\,\sqrt{453}}{74}\right) \\[\SomeSpc]
(9\;7) & \hbox{2.25e-5} & \left( \frac{-53 + \sqrt{5185}}{99}\right.\,,\,&\left.\frac{-30 + 4\,\sqrt{66}}{13}\right) &
(4\;4) & \hbox{5.24e-3} & \left( \frac{-4 + \sqrt{37}}{7}\right.\,,\,&\left.\frac{-3 + \sqrt{13}}{2}\right) \\[\SomeSpc]
(10\;7) & \hbox{1.54e-5} & \left( \frac{-127 + \sqrt{30629}}{250}\right.\,,\,&\left.\frac{-73 + \sqrt{6083}}{26}\right) &
(8\;6) & \hbox{2.73e-5} & \left( \frac{-54 + \sqrt{7057}}{101}\right.\,,\,&\left.\frac{-127 + 7\,\sqrt{453}}{74}\right) \\[\SomeSpc]
(2\;5) & \hbox{1.45e-2} & \left( \frac{-5 + \sqrt{29}}{2}\right.\,,\,&\left.\frac{-1 + \sqrt{2}}{2}\right) &
(2\;3) & \hbox{6.32e-2} & \left( \frac{-3 + \sqrt{13}}{2}\right.\,,\,&\left.\frac{-1 + \sqrt{3}}{2}\right) \\[\SomeSpc]
(3\;8) & \hbox{6.57e-4} & \left( \frac{-23 + \sqrt{629}}{10}\right.\,,\,&\left.\frac{-10 + \sqrt{195}}{19}\right) &
(4\;6) & \hbox{6.9e-4} & \left( \frac{-13 + \sqrt{290}}{11}\right.\,,\,&\left.\frac{-23 + \sqrt{1365}}{38}\right) \\[\SomeSpc]
(3\;7) & \hbox{1.06e-3} & \left( \frac{-9 + \sqrt{101}}{5}\right.\,,\,&\left.\frac{-4 + \sqrt{30}}{7}\right) &
(4\;5) & \hbox{1.72e-3} & \left( \frac{-15 + \sqrt{533}}{22}\right.\,,\,&\left.\frac{-4 + \sqrt{30}}{4}\right) \\[\SomeSpc]
(3\;6) & \hbox{1.98e-3} & \left( \frac{-13 + \sqrt{229}}{10}\right.\,,\,&\left.\frac{-2 + \sqrt{7}}{3}\right) &
(3\;4) & \hbox{9.87e-3} & \left( \frac{-7 + \sqrt{85}}{6}\right.\,,\,&\left.\frac{-3 + 2\,\sqrt{6}}{5}\right) \\[\SomeSpc]
(4\;6) & \hbox{7.42e-4} & \left( \frac{-10 + \sqrt{170}}{14}\right.\,,\,&\left.\frac{-7 + \sqrt{69}}{6}\right) &
(4\;5) & \hbox{1.45e-3} & \left( \frac{-9 + \sqrt{145}}{8}\right.\,,\,&\left.\frac{-8 + 2\,\sqrt{42}}{13}\right) \\[\SomeSpc]
(9\;7) & \hbox{1.03e-5} & \left( \frac{-81 + \sqrt{13226}}{155}\right.\,,\,&\left.\frac{-187 + 3\,\sqrt{4669}}{82}\right) &
(4\;4) & \hbox{3.82e-3} & \left( \frac{-5 + \sqrt{65}}{8}\right.\,,\,&\left.\frac{-11 + \sqrt{221}}{10}\right) \\[\SomeSpc]
(3\;5) & \hbox{4.94e-3} & \left( \frac{-4 + \sqrt{26}}{5}\right.\,,\,&\left.\frac{-2 + \sqrt{6}}{2}\right) &
(5\;5) & \hbox{6.75e-4} & \left( \frac{-13 + 5\,\sqrt{13}}{13}\right.\,,\,&\left.\frac{-2 + \sqrt{10}}{3}\right) \\[\SomeSpc]
(4\;6) & \hbox{8.44e-4} & \left( \frac{-10 + \sqrt{145}}{9}\right.\,,\,&\left.\frac{-19 + 3\,\sqrt{69}}{26}\right) &
(3\;3) & \hbox{2.68e-2} & \left( \frac{-2 + \sqrt{10}}{3}\right.\,,\,&\left.\VSpc{1.2em}{0.3em}-1 + \sqrt{2}\right) \\[\SomeSpc]
(7\;6) & \hbox{1.11e-4} & \left( \frac{-25 + \sqrt{1297}}{48}\right.\,,\,&\left.\frac{-29 + \sqrt{1023}}{13}\right) &
(2\;2) & \hbox{2.04e-1} & \left( \VSpc{1.2em}{0.3em}-1 + \sqrt{2}\right.\,,\,&\left.\frac{-1 + \sqrt{5}}{2}\right) \\[\SomeSpc]
(4\;5) & \hbox{2.45e-3} & \left( \frac{-11 + \sqrt{229}}{18}\right.\,,\,&\left.\frac{-3 + 2\,\sqrt{3}}{2}\right) &
(2\;1)
&
\hbox{3.82e-1}
& \left( \frac{-1 + \sqrt{5}}{2}\right.\,,\,&
\left.1
\VSpc{1.2em}{0.3em}\right]
\\[\SomeSpc]
\hline
\end{array}$}
\end{flushleft}
A sample of matching intervals found as in section \ref{CAmatching}.

\begin{flushleft}
\scalebox{0.90}{$\begin{array}{|cccc|}
\hline
\VSpc{1.4em}{0.3em}
(k_1\; k_2) & \textup{size} & \left(\alpha^-\right., & \left.\alpha^+\right) \\ 
\hline
\VSpc{1.4em}{0.3em}
(257\;257) & \hbox{5.43e-201}  & \left( \hbox to 2cm{\dotfill}\right.\,,\,&\left.\hbox to 2cm{\dotfill}\right)  \\[6pt]
(129\;129) & \hbox{7.27e-101} & \left( \hbox to 2cm{\dotfill}\right.\,,\,&\left.\hbox to 2cm{\dotfill}\right)    \\[6pt]  
(65\;65)   & \hbox{7.98e-51} & \left( \hbox to 2cm{\dotfill}\right.\,,\,&\left.\hbox to 2cm{\dotfill}\right)   \\[6pt] 
(33\;33) & \hbox{ 8.81e-26}  & \left(\begin{array}{l}
\scriptstyle
\left\{\phantom{\sqrt{1}}\kern-9pt
-1051803916417 \right. \\
\scriptstyle
 \left. \;+\; 5\,\sqrt{110424870216034832616745}\kern2pt\right\}/\\
\scriptstyle
\phantom{\sqrt{2}}
1576491320449
\end{array}
\right.,\,&\left. \begin{array}{l}
\\
-1 + \frac{\sqrt{31529826409}}{128045} \\
\\
\end{array}
\right)    \\[8pt]    
(17\;17) & \hbox{2.78e-13} & \left(-1 + \frac{\sqrt{31529826409}}{128045}\right.\,,\,&\left.\frac{-433 + \sqrt{467857}}{649}\right)  \\[6pt] 
(9\;9) & \hbox{5.2e-7} & \left( \frac{-433 + \sqrt{467857}}{649}\right.\,,\,&\left.\frac{-13 + 5\,\sqrt{13}}{13}\right)  \\[6pt] 
(5\;5) & \hbox{6.75e-4} & \left( \frac{-13 + 5\,\sqrt{13}}{13}\right.\,,\,&\left.\frac{-2 + \sqrt{10}}{3}\right)    \\[6pt]
(3\;3) & \hbox{2.68e-2} & \left( \frac{-2 + \sqrt{10}}{3}\right.\,,\,&\left.{-1 + \sqrt{2}}\right)                   \\[6pt]
(2\;2) & \hbox{2.04e-1} & \left( {-1 + \sqrt{2}}\right.\,,\,&\left.\frac{-1 + \sqrt{5}}{2}\right)                    \\[6pt]
\hline
\end{array}$}
\end{flushleft}
A chain of adjacent matching intervals (see section \ref{1accum})

\noindent 
\textsc{Dipartimento di Matematica, Universit\`a di Pisa}, Largo Bruno Pontecorvo 5, 56127 Pisa, Italy. e-mail: carminat@dm.unipi.it \newline
\textsc{Scuola Normale Superiore}, Piazza dei Cavalieri 7, 56123 Pisa, Italy. e-mail: s.marmi@sns.it \newline
\textsc{Scuola Normale Superiore}, Piazza dei Cavalieri 7, 56123 Pisa, Italy. e-mail: a.profeti@sns.it \newline
\textsc{Department of Mathematics, Harvard University}, 1 Oxford St, Cambridge MA 02138, U.S.A. e-mail: tiozzo@math.harvard.edu
\end{document}